\documentclass[11pt]{article}

\usepackage[left=1in, right=1in, top=1in, bottom=1in]{geometry}
\usepackage{amsmath,amssymb,graphicx,amsthm,nicefrac,mathtools}
\usepackage{hyperref}
\usepackage[numbers]{natbib}
\usepackage[english]{babel}
\usepackage[T1]{fontenc}
\usepackage{bbm}
\usepackage{color}
\usepackage{xspace}
\usepackage{statistics-macros}

\newcommand{\privmetric}{\rho_{\rm priv}} 
\newcommand{\dpriv}{d_{\rm priv}} 

\newcommand{\df}[2]{D_f\left({#1} |\!| {#2}\right)} 
\renewcommand{\ball}{\mathbb{B}}

\newcommand{\moment}{k}  
\newcommand{\separation}[1]{\metric^*\!\left({#1}\right)}

\newcommand{\distributions}[1]{\mc{P}_{#1}}
\newcommand{\deficient}{\preceq}
\newcommand{\dominates}{\succeq}
\newcommand{\diffpch}{\diffp_{\mathsf{CH}}}
\newcommand{\deltach}{\delta_{\mathsf{CH}}}

\newcommand{\tvchannels}[1]{\mc{Q}^{\mathsf{TV}}_{#1}}

\newcommand{\indicb}[1]{\mathbf{1}\!\left\{{#1}\right\}}
\newcommand{\indicbs}[1]{\mathbf{1}\!\{{#1}\}}

\title{Privacy and Statistical Risk: \\ Formalisms and Minimax Bounds}

\author{Rina Foygel Barber \\ Department of Statistics \\
  University of Chicago \\ \texttt{rina@uchicago.edu}
  \and John C.\ Duchi \\ Departments of
  Statistics and Electrical Engineering \\
  Stanford University \\ \texttt{jduchi@stanford.edu}}

\begin{document}

\maketitle

\begin{abstract}
  We explore and compare a variety of definitions for privacy and disclosure
  limitation in statistical estimation and data analysis, including
  (approximate) differential privacy, testing-based definitions of privacy,
  and posterior guarantees on disclosure risk. We give equivalence results
  between the definitions, shedding light on the relationships between
  different formalisms for privacy. We also take an inferential perspective,
  where---building off of these definitions---we provide minimax risk bounds
  for several estimation problems, including mean estimation, estimation of
  the support of a distribution, and nonparametric density estimation. These
  bounds highlight the statistical consequences of different definitions of
  privacy and provide a second lens for evaluating the advantages and
  disadvantages of different techniques for disclosure limitation.
\end{abstract}

\section{Introduction}

In this paper, we study several definitions of privacy---formalisms for
limiting disclosure in statistical procedures---and their consequences in
terms of achievable (statistical) risk for estimation and data analysis. We
review (and present a few new) definitions that attempt to capture what,
intuitively, it should mean to limit disclosures from the output of an
inferential task. We focus on several potential definitions for a strong
type of disclosure limitation, where an adversary attempts to glean
information from data released; in particular, notions of privacy centering
around differential privacy (and its relaxations) as formulated
by~\citet{DworkMcNiSm06,DworkKeMcMiNa06}.  As a motivation for the
definitions we study, consider a gene association study with a known list of
subjects; we focus on guarantees such that even if the adversary knows the
disease status (case or control) of many of the subjects in the study, he is
not able to easily identify the disease status of remaining subjects. %
Differential privacy is designed for precisely this setting.

To protect against such an incident, we allow adversaries that are (1)
computationally unbounded and (2) may have access to all elements of a
sample $\{X_1, \ldots, X_n\}$ except for a single unknown observation $X_i$;
the estimators we compute must not release too much information about this
last observation.  While such definitions seem quite strong, they have
motivated a body of work in the cryptography, database, and theoretical
computer science communities, beginning with the work of
\citet*{DworkMcNiSm06} on differential privacy (see also the
papers~\cite{DworkSm09,Smith11,BlumLiRo08,HardtRo10}).  It has been quite
challenging to give rigorous definitions of privacy against weaker
adversaries (such definitions have often been
shown~\cite{DworkMcNiSm06,DworkKeMcMiNa06} to have fatal flaws), but
subsequent works have broadened our understanding of acceptable privacy
definitions and
adversaries~\cite{KasiviswanathanSm13,KiferMa12,BassilyGrKaSm13}.

Our goal in this paper is to make more precise the relationship between
privacy constraints and statistical estimation. Thus, in addition to
presenting a variety of definitions, we provide comparison by focusing on
their consequences for estimation: specifically, we ask whether there are
substantive differences between minimax error for estimating parameters of a
variety of distributions under different definitions of privacy. We show
that, in fact, there are strong commonalities; focusing on mean estimation
to be explicit, we find the minimax mean squared error of estimators under
different privacy constraints is often very similar for seemingly different
definitions of privacy. Nonetheless, some definitions allow more favorable
dependence on dimension than standard (differential) privacy definitions,
though at the expense of some security.

As a consequence of our focus on definitional aspects of privacy and their
effects on statistical estimation and inference problems, we study
estimation of population quantities.
That is, we observe a sample $X_i$, $i = 1, \ldots, n$, drawn from an
unknown distribution $P$, and we wish to make inferences about some
parameter $\theta(P)$ of the data generating distribution $P$ rather
than reporting aspects of the sample itself.
This focus is different from much of the
work on optimality guarantees in private data analysis~\cite{HardtTa10,
  ChaudhuriSaSi12,NikolovTaZh13}, though there have been a few authors who
have studied population quantities (for example, Beimel and
colleagues~\cite{BeimelBrKaNi13,BeimelNiOm08} in the Probably Approximately
Correct (PAC) model for concept learning). For more discussion on the issue
of population estimation in private settings, see the discussion
of~\citet{DuchiJoWa13_focs}.

In Section~\ref{sec:privacy-definitions}, we enumerate a set of potential
definitions of privacy under our adversarial model; these include
differential privacy, approximate differential privacy, a strengthened form
of differential privacy, and several testing-based definitions of privacy.
We follow this in Section~\ref{sec:lower-bounds} by providing minimax lower
bounds for population estimation of mean parameters, distributional support
estimation, and density estimation.  As a concrete example, we consider the
problem of estimating the mean of a $d$-dimensional random variable
$X\in\R^d$ distributed according to a distribution $P$, given a sample
consisting of $n$ i.i.d.\ draws $X_1,\ldots,X_n \simiid P$.  This clearly
shows characteristic effects of dimensionality and moment assumptions---for
instance, bounds on higher moments of $X$, i.e.\ $\E[\ltwo{X}^\moment] <
\infty$ for some fixed $\moment > 1$---on the best possible rates of private
estimation.  Section~\ref{sec:upper-bounds} presents concrete estimation
strategies achieving the minimax lower bounds on mean estimation presented
in Section~\ref{sec:lower-bounds} under our different prvacy definitions.
We will see that our definitions of privacy yield minimax optimal procedures
with nearly the same (asymptotic) dependence on the number $\moment$ of
moments, but different dimension dependence, under squared error loss.

We conclude (in Section~\ref{sec:open-questions}) with some discussion,
including a few avenues for future work. We also present a table
(Table~\ref{table:upper-and-lower-bounds}) summarizing, for $d$-dimensional
mean estimation problems, the effects of the ambient dimension $d$, the
required amount of privacy, and number of moments $k$ assumed for the
distribution from which our data is drawn. This table illustrates the main
consequences of the results in this paper, allowing a more precise
characterization of the tradeoffs between disclosure risk and statistical
performance.

\paragraph{Notation}
Throughout, we use the notation $\channel(\cdot \!\mid x_{1:n})$ to denote
the privatized output channel of the statistician (defined in
Section~\ref{sec:privacy-definitions} below).  We use subscripts $x_{1:n}$
to denote a sequence of $n$ observations $x_1, \ldots,
x_n\in\mc{X}$. Throughout, $x_i$, $x_{1:n}$, $x_i'$, $x_{1:n}'$, etc, denote
fixed value(s) in $\mc{X}$, while $X_i$ and $X_{1:n}$ represent random
variable(s) taking values in $\mc{X}$. We also use $X$ to represent a single
random variable with the same distribution as $X_1,\dots,X_n$ when the
$X_i$ are i.i.d.
The metric $\dham(\cdot,\cdot)$ denotes the Hamming distance between sets;
we have $\dham(x_{1:n},x'_{1:n})\leq 1$
whenever $x_{1:n}$ and $x'_{1:n}$ differ in only one entry.  Given sequences
$a_n$ and $b_n$, we use standard big-$O$ notation, so $a_n = \order(b_n)$
means that $a_n \le c b_n$ for some constant $c<\infty$ and for all suitably
large $n$.  The notation $a_n \lesssim b_n$ means that $a_n \le c b_n$ for
some constant $c<\infty$ and for all $n\geq 1$, while $a_n \ll b_n$ means
that $a_n / b_n \to 0$ as $n \to \infty$.  The function $\indicb{\cdot}$ is
the indicator function, that is, for any event $\mc{E}$, the random variable
$\indicb{\mc{E}}$ is equal to 1 if the event $\mc{E}$ occurs, and 0
otherwise.  We let $a\wedge b = \min\{a,b\}$.

\section{Definitions of privacy}
\label{sec:privacy-definitions}

We are interested in the setting where the adversary has access to all but
one of the observations in the sample: he knows that
$X_1=x_1,\dots,X_{i-1}=x_{i-1},X_{i+1}=x_{i+1},\dots,X_n=x_n$, and seeks to
determine the last remaining observation $X_i$. We represent the
statistician, or estimation procedure, by a channel $\channel(\cdot \!\mid
\! \cdot)$, which, given a sample $X_{1:n}$ drawn from $\mc{X}^n$, releases
a point $\theta \in \Theta$ according to the distribution $\channel(\cdot
\!\mid X_{1:n})$.  Somewhat more formally, a channel is a regular
conditional distribution (or probability kernel)~\cite[e.g.][Chapter
  5]{Kallenberg97} from the sample space $\mc{X}^n$ to the space
$\Theta$. We assume that the data $X_i$ are drawn i.i.d.\ according to some
(unknown) distribution $P$ with a parameter $\theta(P)$ we desire to
estimate, and the goal of the statistician is to release $\theta$ that is as
close as possible to the unknown $\theta(P)$ while guaranteeing that the
adversary cannot identify any one observation $X_i$.  With our privacy goals
in mind, we can consider two related frameworks for bounding the information
available to the adversary:
\begin{enumerate}
\item Likelihood/probability: Under the channel $\channel$, for any region
  $A\subset\Theta$ of the output space, the likelihood of $A$ varies
  minimally for different possible values of $X_i$.
\item \label{item:hypothesis-test}
  Hypothesis testing: If the adversary is considering two possible
  values $x_i$ and $x_i'$ for $X_i$, the channel $\channel$ provides minimal
  power for testing these hypotheses against each other.
\end{enumerate}
In the remainder of this section, we give our definitions of privacy,
beginning with differential privacy, then proceding to hypothesis-testing
variants, and finally showing a variant of privacy that protects against
adaptive and posterior inferences (to be made precise) about the sample. We
make connections between all three via the hypothesis testing
framework~\ref{item:hypothesis-test}.


\subsection{Differential privacy and its cousins}

We begin our presentation of definitions with
differential privacy, due to \citet{DworkMcNiSm06}.
\begin{definition}[Differential privacy]
  \label{definition:differential-privacy}
  A channel $\channel$ is $\diffp$-differentially private ($\diffp$-DP) if
  for all $x_{1:n}$ and $x_{1:n}'$ differing in only one observation,
  \begin{equation*}
    \frac{\channelprob(A \mid \statsample_{1:n})}{\channelprob(A
      \mid \statsample_{1:n}')}
    \le \exp(\diffp)
    ~~~ \mbox{for~all~measurable~} A \subset \Theta.
  \end{equation*}
\end{definition}
\noindent This condition essentially requires that, regardless of the
output, the likelihood under $\channel$ does not distinguish samples
differing in only a small number of observations.

Many schemes for differential privacy actually obey a stronger smoothness
property, and for elucidation we thus define a more stringent version of
privacy, which requires a bounded (semi)metric $\privmetric : \statdomain
\times \statdomain \to \R_+$ defined on the sample space $\statdomain$.  Let
$\privmetric(x, x') \le r$ for all $x, x' \in \mc{X}$ (for example, in a
normed space with norm $\norm{\cdot}$, we may take $\privmetric(x, x') =
\norm{x - x'} \wedge r$). We then define the metric $\dpriv : \statdomain^n
\times \statdomain^n \rightarrow \R_+$ by
\begin{equation*}
  \dpriv(x_{1:n}, x_{1:n}')
  \defeq \frac{1}{r}\sum_{i = 1}^n \privmetric(x_i, x_i').
\end{equation*}
Notably, we have $\dpriv \le \dham$, and we thus define the stronger (more
secure) version of differential privacy we call \emph{smooth differential
  privacy}:
\begin{definition}[Smooth differential privacy]
  \label{definition:smooth-privacy}
  The channel $\channel$ satisfies
  $(\privmetric, \diffp)$-smooth differential privacy if
  for all samples $x_{1:n}$ and $x_{1:n}'$,
  \begin{equation}
    \label{eqn:strong-diffp}
    \frac{\channel(A \mid x_{1:n})}{\channel(A \mid x_{1:n}')}
    \le \exp\left(\diffp \dpriv(x_{1:n}, x_{1:n}')\right)
    ~~~ \mbox{for~all~measurable~} A \subset \Theta.
  \end{equation}
\end{definition}

Both of these definitions are quite strong: they require a likelihood ratio
bound to hold even for an extremely low probability event $A$.
\begin{example}
  \label{ex:dp-strong}
  Suppose that $X_i \in [0, 1]$ for all $i$, and we release the mean
  corrupted by an independent $\normal(0, \stddev^2 / n)$ variable, that is,
  $\theta = \frac{1}{n} \sum_{i=1}^n X_i + W$, where $W \sim \normal(0,
  \stddev^2 / n)$.  Then the channel densities satisfy the ratios
  \begin{equation*}
    \exp\left(-\frac{|\theta|}{2\stddev^2}
    - \frac{1}{2\stddev^2}\right)
    \le \frac{\channeldens(\theta \mid x_{1:n})}{
      \channeldens(\theta \mid x_{1:n}')}
    \le \exp\left(\frac{|\theta|}{2\stddev^2} +
    \frac{1}{\stddev^2}\right),
  \end{equation*}
  which fails to be $\diffp$-DP for any $\diffp$,
  as we may have $|\theta|>\stddev^2\diffp$. Yet the
  probability of releasing such a large $\theta$ is exponentially small
  in $n$.
\end{example}

\noindent
In this example, the probability of releasing a large $\theta$---thus
revealing information distinguishing the samples $X_{1:n}$ and
$X_{1:n}'$---is negligible under both samples.  Intuitively, these extremely
low probability events should not cause us to declare a channel
non-private.  Such situations motivated \citet{DworkKeMcMiNa06} to define a
relaxed version of differential privacy disregarding low-probability events:
\begin{definition}
  \label{definition:approximate-dp}
  A channel $\channel$ is $(\diffp,\delta)$-approximately
  differentially private ($(\diffp,\delta)$-DP) if, for all $x_{1:n}$ and
  $x'_{1:n}$ differing in only one observation,
  \begin{equation*}
    \channel(A \mid x_{1:n})
    \leq e^{\diffp} \cdot \channel (A \mid x'_{1:n}) + \delta
    \text{ for all measurable } A\subset\Theta.
  \end{equation*}
\end{definition}
\noindent
For approximate differential privacy, as in differential privacy, one
typically thinks of $\diffp$ as a constant (or decreasing
polynomially to 0 as $n \to \infty$). To protect against catastrophic
disclosures, one usually assumes that $\delta$ decreases super-polynomially,
though not exponentially, to zero, that is, that that $\delta \le
\exp(-p(n))$ where $p(n)$ is a function satisfying $\log n \ll p(n) \ll
n$. While the relaxed conditions of approximate differential privacy
address situations such as Example~\ref{ex:dp-strong}, we show in
Section~\ref{sec:lower-bounds} that the consequences for estimation under
each of the privacy
definitions~\ref{definition:differential-privacy},
\ref{definition:smooth-privacy}, and \ref{definition:approximate-dp} are
quite similar.

\subsection{Testing-based  and  divergence-based definitions of privacy}

We now turn to alternate definitions of privacy, again considering an
adversary who knows most of the data in the sample, but we build on a
framework of hypothesis testing.  We believe these variants both give some
intuition for the definitions of disclosure limitation and suggest potential
weakenings of
Definitions~\ref{definition:differential-privacy}--\ref{definition:approximate-dp}.
Our first observation, essentially noted by \citet[Thm.~2.4]{WassermanZh10}
due to \citet{OhVi13}, is that differential privacy is equivalent to a form
of false negative and false positive rate control for hypothesis tests that
distinguish samples $X_{1:n}$ and $X_{1:n}'$ differing in a single
observation.  In particular, let us assume that a test $\test : \Theta \to
\{0, 1\}$ tries to distinguish the following two hypotheses, where $X_{1:n}$
is known except for its $i$th entry:
\begin{equation*}
  H_0 ~ : ~ X_{1:n} = (x_1,\dots,x_{i-1},x_i, x_{i+1},\dots,x_n)
  ~~~~ \mbox{and} ~~~~
  H_1 ~ : ~ X_{1:n} = (x_1,\dots,x_{i-1},x_i', x_{i+1},\dots,x_n).
\end{equation*}
Here a result of 0 from the test $\test$ indicates evidence that $X_i =
x_i$, while a 1 indicates evidence instead that $X_i = x'_i$.  For shorthand
let $\channel(\cdot \!\mid H_j)$ denote the channel (private)
distribution under $H_j$.  We have the following
result; we provide a proof for completeness in
Sec.~\ref{sec:proof-diffp-test-errors}.
\begin{proposition}
  \label{proposition:diffp-test-errors}
  A channel $\channel$ satisfies $(\diffp, \delta)$-approximate differential
  privacy if and only if for all hypothesis tests $\test$ mapping to $\{0,
  1\}$, for any $x_{1:n},x'_{1:n}$ with $\dham(x_{1:n},x'_{1:n})\leq 1$,
  \begin{equation}
    \label{eqn:hypothesis-test-dp}
    \channel(\test = 1 \mid H_0) + e^\diffp\cdot \channel(\test = 0 \mid H_1)
   \ge 1 - \delta,
  \end{equation}
  where we define hypotheses $H_0:X_{1:n}=x_{1:n}$ and
  $H_1:X_{1:n}=x'_{1:n}$.  Moreover, $(\diffp, \delta)$-approximate
  differential privacy implies
  \begin{equation}
    \label{eqn:disclosure-risk}
    \channel(\prox = 1 \mid H_0) + \channel(\prox = 0 \mid H_1)
    \ge \frac{2}{1 + e^\diffp} - \frac{\delta}{1 + e^\diffp}
    \geq 1 - \frac{\diffp}{2}
    - \frac{\delta}{2}.
  \end{equation}
\end{proposition}

That is, for small $\diffp$, the sum of the false positive rate and false
negative rate, when testing $H_0$ against $H_1$, is nearly $1$ under
differential privacy (this is similarly true for approximate differential
privacy).  This suggests a potential weakening of differential privacy: can
we require that the adversary cannot test $H_0$ against $H_1$ with any high
power?  That is, will we achieve sufficient protection if we base privacy on
mechanisms that achieve disclosure risk bounds of the
form~\eqref{eqn:disclosure-risk}?\footnote{There is also a Bayesian
  interpretation of differential privacy~\cite{KasiviswanathanSm13} that
  says that an adversaries prior and posterior beliefs after observing the
  output of $\channel$ cannot change much; we defer discussion to
  Section~\ref{sec:open-questions}.}

As a first approach, we note that Le Cam's inequality~\cite[e.g.][Chapter
  2.4]{Tsybakov09} implies that for any distributions $P_0$ and $P_1$, we
have
\begin{equation}
  \label{eqn:le-cam}
  \inf_{\test} \left\{P_0(\test \neq 0) + P_1(\test \neq 1)\right\}
  = 1 - \tvnorm{P_0 - P_1},
\end{equation}
where the infimum is taken over all measurable functions,
and we recall that the total variation distance
is $\tvnorm{P_0 - P_1} = \sup_A |P_0(A) - P_1(A)| = \half
\int |dP_0 - dP_1|$.
  Based on Le Cam's
inequality~\eqref{eqn:le-cam} and the
consequence~\eqref{eqn:disclosure-risk} of differential privacy, we arrive
at the following proposal for privacy, which bounds differences rather than
ratios of likelihoods:
\begin{definition}
  \label{definition:tv-private}
  A channel $\channel$ is $\diffp$-total variation private
  ($\diffp$-TVP) if, for all $x_{1:n}$ and $x'_{1:n}$ differing in only one
  observation,
  \begin{equation*}
    \left|\channel(A \mid x_{1:n})-
    \channel(A \mid x'_{1:n})\right|
    \leq \diffp\text{ for all
    }A\subset\Theta\;.
  \end{equation*}
  Equivalently, the error for testing $x_{1:n}$ against $x_{1:n}'$
  has lower bound
  \begin{equation}
    \label{eqn:tv-private}
    \inf_\test \left\{\channel(\test = 1 \mid x_{1:n})
    + \channel(\test = 0 \mid x_{1:n}')\right\} \ge 1 - \diffp.
  \end{equation}
\end{definition}
\noindent Here the notation $\channel(\test = 1 \mid x_{1:n})$ is shorthand
for $\channel(\{\theta:\test(\theta) = 1\} \mid x_{1:n})$.


Notably, $\diffp$-total variation privacy means that an adversary cannot
accurately test between $x_{1:n}$ and $x_{1:n}'$. Comparing
inequality~\eqref{eqn:tv-private} with
inequality~\eqref{eqn:disclosure-risk}, we see that
$\diffp$-TV privacy is less stringent than differential privacy.
Unfortunately, while differential privacy may be strong, the testing-based
weakening~\eqref{eqn:tv-private} may not be fully satisfactory, as the
following well-known example (e.g.~\cite{Dwork08}) shows:
\begin{example}[``Release one at random'']
  \label{ex:release-one}
  Consider a channel $\channel$ that selects one
  observation at random and releases it,
  so that $\channel(A \mid x_{1:n}) = \frac{1}{n}\sum_{i=1}^n \indicb{x_i\in A}$.
  Here the sample space and output space are equal, $\mc{X} = \Theta$.
  When samples $x_{1:n}$ and $x'_{1:n}$ differ only
  at position $i$, then
  $|\channel(A \mid x_{1:n}) - \channel(A \mid x'_{1:n})| \leq \frac{1}{n}$
  for any set $A \subset \mc{X}$, so $\channel$
  is $\diffp$-TV private for any $\diffp \ge 1/n$.
\end{example}

\noindent
While pathological, the constructed channel is clearly not private in an
intuitive sense---one of the $n$ individuals
in the sample will suffer a complete loss of privacy,
even though initially each individual had only a small chance
of having his data revealed. Thus, simple hypothesis testing variants of
privacy, such as inequality~\eqref{eqn:disclosure-risk} (and the equivalent
total variation privacy of Definition~\ref{definition:tv-private}) do not
provide sufficient protection against disclosure risk.  One way to address
this problem is to impose stronger divergence requirements on the channels
$\channel$ in Definition~\ref{definition:tv-private}, for example, choosing
a measure between distributions that is infinite when they are not mutually
absolutely continuous.

Before stating this extension of total-variation privacy, we recall that for
a convex function $f : [0, \infty] \to \R \cup \{+\infty\}$ satisfying
$f(1) = 0$, the associated $f$-divergence between distributions $P$ and $Q$
is defined as
\begin{equation*}
  \df{P}{Q} = \int f\left(\frac{dP}{dQ}\right) dQ
  = \int f\left(\frac{p}{q}\right) q d\mu,
\end{equation*}
where $\mu$ denotes a measure with respect to which $P$ and $Q$ are
absolutely continuous (with densities $p$ and $q$).
For such a convex $f$, we define
\begin{definition}[Divergence privacy]
  \label{definition:divergence-privacy}
  The channel $\channel$ is $\diffp$-$f$-divergence private if
  \begin{equation*}
    \sup\left\{\df{\channel(\cdot \!\mid \statsample_{1:n})}{
      \channel(\cdot \!\mid \statsample_{1:n}')}
    \mid \dham(\statsample_{1:n}, \statsample_{1:n}') \le 1
    \right\} \le \diffp.
  \end{equation*}
\end{definition}
\noindent
We recover Definition~\ref{definition:tv-private} by taking $f(t) = |t -
1|$, we may take $f(t) = t\log t$ to obtain $\diffp$-Kullback-Leibler
($\diffp$-KL) privacy (which is more stringent than TV-privacy by Pinsker's
inequality):
\begin{equation}
  \label{eqn:kl-private}
  \sup\left\{ \dkl{\channel(\cdot \!\mid \statsample_{1:n})}{
    \channel(\cdot \!\mid \statsample_{1:n}')}
  \mid \dham(\statsample_{1:n}, \statsample_{1:n}') \le 1,
  \statsample_i \in \statdomain
  \right\} \le \diffp.
\end{equation}
We show in the sequel that mechanisms $\channel$ satisfying KL-privacy (and
hence TV-privacy) \emph{can} yield more accurate estimates than
approximately differentially private mechanisms.  In contrast to these two
divergence-based definitions, however, differential privacy offers ``privacy
in hindsight,'' where even after the channel releases its output, each
individual's privacy is relatively secure.

\subsection{Conditional hypothesis testing privacy}
The ``release-one-at-random'' example highlights a need for stronger privacy
requirements than hypothesis testing privacy (or equivalently, total
variation privacy).  With this in mind, we turn to a more restrictive notion
of privacy based on hypothesis testing, where we assess the accuracy of a
hypothesis test $\test$ \emph{conditional} on the output. This inspires an
extension of our hypothesis testing idea that conditions on the observed
output of the channel.

To define this notion of conditional hypothesis testing, we require a few
additional definitions.  We write $\distributions{\mc{Y}}$ to denote the set
of all distributions on the space $\mc{Y}$ (treating the $\sigma$-algebra as
implicit), and given two spaces $\mc{X}$ and $\mc{Y}$, we abuse notation and
write $\channel : \mc{X}^n \to \distributions{\mc{Y}}$ to denote that
$\channel$ is a regular conditional probability for $Y$ taking values in
$\mc{Y}$ given $X_{1:n} \in \mc{X}^n$, that is, $\channel(\cdot \!\mid
x_{1:n})$ is a probability distribution on $\mc{Y}$ for each $x_{1:n} \in
\mc{X}^n$ and is $\mc{X}^n$-measurable (a Markov kernel from $\mc{X}^n$ to
$\mc{Y}$). With this notation, we define channel composition
as follows.
\begin{definition}[Composition of channels]
  \label{definition:composition}
  Given channels $\channel : \mc{X}^n \rightarrow \distributions{\mc{Y}}$ and
  $\channel' : \mc{Y} \rightarrow \distributions{\mc{Z}}$, the
  \emph{composition} of $\channel'$ with $\channel$, denoted $\channel'\circ
  \channel:\mc{X}^n\rightarrow \distributions{\mc{Z}}$ is defined via the
  hierarchical model
  \begin{equation*}
    Y\mid X_{1:n} \sim \channel(\cdot \!\mid X_{1:n})
    ~~~ \mbox{and} ~~~
    Z\mid Y,X_{1:n} \sim \channel'(\cdot \!\mid Y).
  \end{equation*}
\end{definition}
\noindent
That is, we view $\channel'$ as a stochastic kernel
from the set $\mc{Y}$ to the set $\mc{Z}$, where
\begin{equation*}
  (\channel' \circ \channel)(A \mid x_{1:n})
  = \int_y \channel'(A \mid y) d\channel(y \mid x_{1:n}).
\end{equation*}
With this definition of composition, we give a definition capturing when a
channel communicates less than another, which also provides a partial
order on channels.
\begin{definition}
  \label{definition:deficient-channel}
  Given channels $\channel:\mc{X}^n \rightarrow \distributions{\mc{Y}}$ and
  $\channel':\mc{X}^n\rightarrow\distributions{\mc{Z}}$, we say
  \emph{$\channel'$ is less informative than $\channel$}, written $\channel'
  \deficient \channel$, if there exists a channel
  $\channel'':\mc{Y}\rightarrow\distributions{\mc{Z}}$ such that
  $\channel'=\channel''\circ \channel$.
\end{definition}
\noindent
The definition coincides with the notion of {\em deficiency} arising in the
literature on statistical inference and comparison of experiments, dating to
Blackwell's work in the 1950s (see, for example, \citet[Chapter
  2]{LeCamYa00}, or \citet[Section VI]{LieseVa06}).

Definition~\ref{definition:deficient-channel} is natural: as we construct
$\channel'$ from $\channel$ via an independent randomization, no new
information about the sample $X_{1:n}$ arises by moving from $\channel$ to
$\channel'$.
Indeed, any channel $\channel' \deficient \channel$ inherits privacy
properties of $\channel$; further processing cannot increase disclosure
risk.  More specifically, we have an information processing inequality
(cf.~\cite[Chapter 2]{LieseVa06,CoverTh06}; see
Section~\ref{sec:proof-information-processing} for a proof).
\begin{observation}[Information processing]
  \label{observation:information-processing}
  If $\channel \dominates \channel'$, then:
  \begin{enumerate}[(1)]
  \item If $\channel$ is $\diffp$-$f$-divergence private
    (Definition~\ref{definition:divergence-privacy}) then $\channel'$ is
    $\diffp$-$f$-divergence private.
  \item If $\channel$ is $(\diffp,\delta)$-differentially private, then
    $\channel'$ is $(\diffp,\delta)$-differentially private.
  \end{enumerate}
\end{observation}

Using the notion of deficiency, we now provide a strengthened version of
testing-based privacy.

\begin{definition}
  \label{def:chtp}
  A channel $\channel$ is $(\diffp,\delta)$-conditional hypothesis testing
  private (CHTP) if for any pair of samples $x_{1:n}$, $x'_{1:n}$ with
  $\dham(x_{1:n},x'_{1:n})\leq 1$, any set $A\subset\Theta$ satisfying
  $\channel(A\mid x_{1:n})\wedge\channel(A\mid x'_{1:n})\ge\delta$, and any
  test $\test:\mc{Y}\rightarrow\{0,1\}$, we have
  \begin{equation}
    \label{eqn:conditional-test}
    {\channel}(\test = 1
    \mid x_{1:n}; \theta \in A)
    + {\channel}(\test = 0 \mid x'_{1:n}; \theta \in A) \ge 1 - \diffp,
  \end{equation}
  where the conditional channel is defined as
  \begin{equation*}
    \channel(B\mid x_{1:n}; \theta \in A)
    \defeq
    \frac{\channel(B\cap A\mid x_{1:n})}{\channel(A\mid x_{1:n})}.
  \end{equation*}
\end{definition}
\noindent
We make a few remarks on this definition. It says that the channel
$\channel$ must have large probability of error in testing between samples
$x_{1:n}$ and $x'_{1:n}$, even conditional on the output $\theta$ of the
channel (at least for $\theta$ in sets with high enough probability).  The
definition is nontrivial only for $\diffp < 1$.  Unlike the notions of
privacy introduced earlier (DP, TVP, and HTP), which are inherited
(Observation~\ref{observation:information-processing}),
CHTP is not inherited---there exist channels $Q'\preceq Q$ where $Q$ is
$(\diffp,\delta)$-CHTP while $Q'$ is not.


While expression~\eqref{eqn:conditional-test} is superficially similar to
our earlier testing-based definitions of privacy, its reliance on the
conditioning set $A$ is important. It addresses the criticism of our
original definition of testing-based privacy
(cf.\ Example~\ref{ex:release-one}), which only provides \emph{a priori}
protection. This new definition says that even \emph{after} observing the
output of the channel $\channel$, it is hard to test accurately between
samples $x_{1:n}$ and $x'_{1:n}$ differing in only a single entry; this
posterior protection is more substantial.  Another way to interpret
posterior privacy, as compared to {\em a priori} privacy, is that we would
like to limit the accuracy of hypothesis tests even when the hypotheses
$H_0$ and $H_1$ and the test $\test$ are constructed adaptively upon
observing the output of the channel $\channel$.  To contrast with our
earlier definitions, recall Example~\ref{ex:release-one}
(``release-one-at-random''). Under the release-one channel, we have little
power {\em a priori} to test hypotheses $H_0:X_i=x_i$ and $H_1:X_i=x_i'$
against each other, since it is unlikely (probability $\frac{1}{n}$) that
the $i$th data point will be released. However, writing
$i_{\text{released}}$ to denote the index of the randomly released data
point, we are able to test hypotheses about $X_{i_{\text{released}}}$ with
perfect accuracy. Requiring conditional hypothesis testing privacy, on the
other hand, accounts for this issue and does not allow the
``release-one-at-random'' mechanism.

Interestingly, we can show that Definition~\ref{def:chtp}
is essentially equivalent to (approximate) differential
privacy, once we account for the issue of ``inheritance'' of the CHTP property:

\begin{theorem}
  \label{theorem:chtp-implies-diffp}
  If $\channel$ is $(\diffp,\delta)$-DP, then it is $(\diffpch,\deltach)$-CHTP
  where
  \begin{equation*}
    \diffpch = (1+e^{-\diffp})\cdot (1-e^{-2\diffp}) \leq 4 \diffp
    ~~~\mbox{and}~~~
    \deltach =  (e^{2\diffp}-e^{\diffp})^{-1}\delta
    \leq \diffp^{-1}\delta.
  \end{equation*}
Conversely,  suppose that for some $\diffpch<1$ and $\deltach$,
 $\channel'$ is $(\diffpch, \deltach)$-CHTP for every $\channel'\deficient\channel$. Then
 $\channel$ is $(\diffp, \delta)$-DP with
  \begin{equation*}
    \diffp =2 \log \left(\frac{1 + \diffpch}{1 -  \diffpch}\right)
    \le \frac{4\diffpch}{1 - \diffpch}
    ~~~\text{and}~~~
    \delta= \deltach \cdot \frac{1+\diffpch}{1-\diffpch}.
  \end{equation*}
\end{theorem}
\noindent
See Appendix~\ref{sec:proof-chtp-implies-diffp} for a proof of
this theorem.

We have thus come full circle: differential privacy appears to be a strong
requirement, so the simple \emph{a priori} variants of testing-based privacy
may seem more natural, requiring only that the chances of discovering any
particular person in a dataset are small. However, the
``release-one-at-random'' example motivates us to move away from {\em a
  priori} privacy towards the posterior privacy guaranteed by the new notion
of conditional hypothesis testing---which is equivalent to differential
privacy.

\section{Lower bounds on estimation of population quantities}
\label{sec:lower-bounds}

Essential to any proposed definition of privacy or disclosure risk is an
understanding of the fundamental consequences for inferential procedures and
estimation.  With that in mind, and having provided several potential
privacy definitions, we turn to an elucidation of some of their
estimation-theoretic consequences. In particular, we consider minimax risk,
defined as follows. Let $\mc{P}$ denote a family of distributions
supported on a set $\mc{X}$, and let $\theta : \mc{P} \to \Theta$ denote a
population quantity of interest.  We also require an error metric with which
to measure the performance of an estimator; to that end, we let $\metric :
\Theta \times \Theta \to \R_+$ denote a (semi)metric on the space $\Theta$,
and let $\loss : \R_+ \to \R$ be a loss function.  For a fixed
privacy-preserving channel $\channel$, the maximum error for estimation of
the population parameter $\theta(P)$ is
\begin{equation}
  \label{eqn:minimax-risk-Q}
  \minimax_n(\theta(\mc{P}), Q, \loss \circ \metric)
  \defeq
  \sup_{P \in \mc{P}}
  \E_{P,Q}\left[\loss(\metric(\what{\theta}(X_{1:n}), \theta(P)))\right],
\end{equation}
where the expectation is taken over both the sample $X_{1:n}$ and the
estimator $\what{\theta}(X_{1:n})$.  To be precise, the data $X_1,\dots,X_n$
are drawn i.i.d.\ from the distribution $P$, then
the estimator $\what{\theta}(X_{1:n})$ is drawn according to
the channel $\channel(\cdot \!\mid X_{1:n})$ conditional on $X_{1:n}$.

We are interested in minimizing this error over all possible
privacy-preserving mechanisms, so that for a family $\mc{Q}$ (i.e.\ the
set of channel distributions $Q$ satisfying some chosen
definition of privacy), we study the \emph{minimax risk} for
estimation of the population parameter $\theta(P)$, defined as
\begin{equation}
  \minimax_n(\theta(\mc{P}), \mc{Q}, \loss \circ \metric)
  \defeq \inf_{Q \in \mc{Q}} \minimax_n(\theta(\mc{P}), Q, \loss \circ
  \metric) = \inf_{Q \in \mc{Q}}
  \sup_{P \in \mc{P}}
  \E_{P,Q}\left[\loss(\metric(\what{\theta}(X_{1:n}), \theta(P)))\right].
  \label{eqn:minimax-def}
\end{equation}

As a concrete example, consider the problem of mean estimation over the
space $\mc{X}=\R^d$.  After drawing the sample $X_1,\dots,X_n$ i.i.d.\ $P$,
we would like to estimate the mean of this distribution. Consider the family
$\mc{P}$ of all distributions $P$ such that $\E_P[\ltwo{X}^2] \le 1$, and
let $\theta(P) = \E_P[X] \in \R^d$ be the mean of $P$ (the parameter we wish
to estimate).  In this setting, we might use mean-squared error, taking
$\loss(t) = t^2$ and $\metric(\theta, \theta') = \ltwo{\theta -
  \theta'}$. The minimax risk is then
\begin{equation*}
  \minimax_n(\theta(\mc{P}), \mc{Q}, \loss \circ \metric)
  = \inf_{Q \in \mc{Q}}
  \sup_{P : \E_P[\ltwo{X}^2]\le 1}
  \E_{P,Q}\left[\ltwo{\what{\theta}(X_{1:n}) - \E_P[X]}^2\right].
\end{equation*}
Our goal, for the remainder of this section, is to find lower bounds on this
minimax error (both for the mean estimation problem and the general setting)
under each of the privacy frameworks in the prequel. In
Section~\ref{sec:upper-bounds}, we derive upper bounds on the minimax
error for mean estimation via concrete constructions of private channels
$\channel$ under the various frameworks.

\paragraph{Our approach}
A standard route for lower bounding the minimax
risk~\eqref{eqn:minimax-risk-Q} is to reduce the estimation problem to a
testing problem, where we aim to identify a point $\optvar \in \optdomain$
from a finite collection of well-separated points~\cite{Yu97, YangBa99}.
Given an index set $\packset$ of finite cardinality, the indexed family of
distributions $\{\statprob_\packval, \packval \in \packset\} \subset \mc{P}$
is said to be a $2\delta$-packing of $\optdomain$ if
$\metric(\optvar(\statprob_\packval), \optvar(\statprob_{\altpackval})) \ge
2 \delta$ for all $\packval \neq \altpackval\in\packset$.

\newcommand{\channels}{\mc{Q}}

In the standard hypothesis testing problem (without privacy constraints),
nature chooses $\packrv \in \packset$ uniformly at random, then (conditional
on $\packrv = \packval$) draws a sample $\statrv_1, \dots, \statrv_\numobs$
i.i.d.\ from the distribution $\statprob_\packval$; the problem is to
identify the member $\packrv$ of the packing set $\packset$. Several
techniques exist for lower bounding the risk of this testing problem (see,
for example, \citet{Yu97}, \citet{Tsybakov09}, or \citet{YangBa99} for a
survey of such techniques). In short, however, under the $2\delta$-packing
construction above, each begins with the classical reduction of estimation
to testing that
\begin{equation*}
  \minimax_n(\theta(\mc{P}), \loss \circ \metric)
  \defeq \inf_{\what{\theta}} \sup_{P \in \mc{P}}
  \E_P[\loss(\metric(\what{\theta}(X_{1:n}), \theta(P)))]
  \ge \loss(\delta) \inf_{\test} \P(\test(\statrv_{1:n}) \neq \packrv),
\end{equation*}
where $\P$ denotes probability 
under the joint distribution of both $\packrv$ and the sample
$\statrv_{1:n}$.
In this work, we have the additional complication that,
instead of releasing a parameter $\what{\theta}$ computed directly
on the sample $X_{1:n}$, our estimator must satisfy some type of
privacy---it is drawn from some channel $\channel(\cdot \!\mid\! \cdot)$.
In particular, we immediately have the following extension of the
classical lower testing-based lower bound: for any $2\delta$-packing
and family $\channels$ of channels,
\begin{equation}
  \label{eqn:estimation-to-testing}
  \minimax_n(\theta(\mc{P}), \channels, \loss \circ \metric)
  \ge \loss(\delta) \inf_{\channel \in \channels} \inf_\test
  \P_{\channel}\left(\test(\what{\theta}(X_{1:n})) \neq \packrv\right),
\end{equation}
where $\what{\theta}$ is distributed according to $\channel(\cdot \mid
X_{1:n})$ and $\P_\channel$ denotes the joint distribution of the packing
index $\packrv$, sample $X_{1:n}$, and $\what{\theta}$ drawn from
$\channel$. In the next sections, we show how to derive lower bounds for
this more complex problem.

\subsection{Lower bounds for weak forms of privacy}

We begin by focusing on private estimation under the weakest privacy setting
we have defined: the $\diffp$-$f$-divergence privacy settings (recall
Definitions~\ref{definition:tv-private}
and~\ref{definition:divergence-privacy}).  In particular, we prove all
results in this subsection using $\diffp$-total variation privacy; this is,
in a sense, the smallest $f$-divergence (cf.~\citet[Section V]{LieseVa06},
where it is shown that all $f$-divergences can be written as mixtures of
variation-like distances) and thus the weakest form of privacy. The lower
bounds we prove here extend immediately to all the definitions of privacy in
this paper, as all the variants of differential privacy
(Definitions~\ref{definition:differential-privacy},
\ref{definition:smooth-privacy}, \ref{definition:approximate-dp}) and
KL-divergence privacy~\eqref{eqn:kl-private} imply total variation privacy.

For a channel $\channel$, the information available to an observer about the
original distribution $P$ of the data is disguised via $\channel$.  To that
end, for a channel $\channel$ and distribution $\statprob_\packval$, we
define the marginal
\begin{equation*}
  \marginprob^n_\packval(A) \defeq \int_{\statdomain^n}
  \channel(A \mid \statsample_{1:n}) d\statprob^n_\packval(\statsample_{1:n}),
\end{equation*}
where $\statprob_\packval^n$ is the $n$-fold product distribution (that is,
$X_{1:n}\sim \statprob_\packval^n$ is equivalent to
$X_1,\dots,X_n \simiid \statprob_\packval$).  This is the
marginal distribution of the privately released estimator $\what{\theta}$
when the initial sample $X_{1:n}$ is drawn from $\statprob_\packval^n$.

Our first set of lower bounds builds on Le Cam's two-point method
(cf.~\cite{Yu97}), which relates minimax errors directly to a binary
hypothesis test; in this case, the packing set consists only of the
distributions $P_0$ and $P_1$, each chosen with probability $\half$.  Let
$\test : \Theta \to \{0, 1\}$ denote an arbitrary testing procedure based on
the output of the private mechanism $\channel$.  Then Le Cam's method gives
the following lower bound (recall expression~\eqref{eqn:le-cam}; see
also~\cite[Theorem 2.2]{Yu97,Tsybakov09}), where in the lemma $\P$ denotes
the joint distribution of the random variable $\packrv \in \{0, 1\}$, the
sample $X_{1:n} \simiid P_\packval$ conditional on $\packrv = \packval$, and
the estimator $\what{\theta} \sim \channel(\cdot \!\mid X_{1:n})$.
\begin{lemma}[Le Cam's method]
  \label{lemma:le-cam}
   For the binary test described above, the probably
   of making an error is lower bounded as
  \begin{equation*}
    \inf_\test \P\left(\test(\what{\theta}) \neq \packrv\right)
    = \half - \half
    \tvnorm{\marginprob^n_0 - \marginprob^n_1},
  \end{equation*}
 where the infimum is taken over all testing procedures.
\end{lemma}
\noindent

With this result in mind, if we can prove that the marginals
$\marginprob_\packval$ are substantially closer in variation distance than
are the $P_\packval$, we may obtain sharper minimax lower bounds on
estimation.  To that end, we prove the following quantitative data
processing inequality, which says that for small privacy parameter $\diffp$
(i.e.~a high privacy level), the output of the channel contains relatively
little information about the true distribution $P$. (See
Sec.~\ref{sec:proof-contraction} for a proof.)
\begin{theorem}
  \label{theorem:contraction}
  Let $\statprob_0$ and $\statprob_1$ be
  probability distributions on $\statdomain$ and $\statprob_\packval^n$,
  $\packval \in \{0, 1\}$, be
  their $n$-fold products. Under $\diffp$-total-variation privacy
  (definition~\ref{definition:tv-private}),
  \begin{equation*}
    \tvnorm{\marginprob_0^n - \marginprob_1^n} \le
    \left(2 n \diffp \tvnorm{\statprob_0 - \statprob_1}\right)
    \wedge \tvnorm{\statprob_0^n - \statprob_1^n}.
  \end{equation*}
\end{theorem}
\noindent
We now give two applications of this contraction inequality to classical
estimation problems.

\subsubsection{Mean estimation under total variation privacy}

As our first example, we consider estimation of a mean with mean-squared
error as our metric, studying a natural family of distributions for this
setting. In particular, we consider families of distributions with
conditions on their moments that allow efficient estimation in non-private
settings. In particular, define the family $\mc{P}_\moment^d(r)$ of
distributions supported on $\R^d$ by
\begin{equation}
  \mc{P}_\moment^d(r) \defeq \left\{ P
  \mid \supp P \subset \R^d
  ~ \mbox{and} ~ \E_P[\ltwo{X}^\moment] \le r^\moment \right\},
  \label{eqn:moment-distributions}
\end{equation}
that is, distributions on $\R^d$ with $\moment$th moment bounded by
$r^\moment$.  For any $\moment \ge 2$, the minimax mean-squared
error~\eqref{eqn:minimax-def} for mean estimation in this family is bounded
by $r^2 / n$ when there are no privacy constraints. Indeed, taking the
sample mean as our estimator gives risk bounded as
\begin{equation*}
  \E\bigg[\ltwobigg{\frac{1}{n} \sum_{i=1}^n X_i - \E[X]}^2\bigg]
  \le \frac{1}{n} \E[\ltwo{X - \E[X]}^2]
  \le \frac{1}{n} \E[\ltwo{X}^2]
  \le \frac{1}{n} \E[\ltwo{X}^\moment]^{\frac{2}{\moment}}
  \le \frac{r^2}{n}.
\end{equation*}
However, after adding a privacy constraint, we have the
following result, which is a
consequence of inequality~\eqref{eqn:estimation-to-testing},
Lemma~\ref{lemma:le-cam}, and Theorem~\ref{theorem:contraction}.
\begin{proposition}
  \label{proposition:minimax-lower-tv}
  Consider the problem of mean estimation over the
  class~\eqref{eqn:moment-distributions} of distributions. If
  $\tvchannels{\diffp}$ denotes the family of $\diffp$-TV-private channels,
  then
  \begin{equation}
    \label{eqn:tv-lower-bound}
    \minimax_n(\theta(\mc{P}_\moment^d(r)), \tvchannels{\diffp}, |\cdot|^2)
    \gtrsim \frac{r^2}{n}
    + r^2\left(\frac{1}{\diffp^2 n^2}\right)^{\frac{\moment - 1}{\moment}}.
  \end{equation}
\end{proposition}
\begin{proof}
  We apply Le Cam's method and the lower
  bound~\eqref{eqn:estimation-to-testing}.  First, we fix $\delta > 0$ (to
  be chosen later), and we define the distributions $P_0$ and $P_1$ on $\{-r
  \delta^{-1/\moment}, 0, r\delta^{-1/\moment}\}$ via
  \begin{equation*}
    P_0(X = -r\delta^{-1/\moment}) = \delta,
    ~~~ P_1(X = r\delta^{-1/\moment}) = \delta,
    ~~~ \mbox{and} ~~~
    P_0(X = 0) = P_1(X = 0) = 1 - \delta.
  \end{equation*}
  Their respective means are $\theta_0 = \E_{P_0}[X]
  = -r \delta^{1-1/\moment}$ and $\theta_1 = r \delta^{1-1/\moment}$, while
  $\tvnorm{P_0 - P_1} = \delta$ and $\E[|X|^\moment] = r^\moment
  (\delta^{-1/\moment})^\moment \delta = r^\moment$. Via
  the estimation-to-testing bound~\eqref{eqn:estimation-to-testing},
  we have the
  following minimax lower bound for any $\diffp$-total variation private
  channel $\channel$:
  \begin{align*}
    \minimax_n(\theta(\mc{P}_\moment^1(r)), Q, |\cdot|^2)
    & \ge
    \half \left(r\delta^{1-1/k}\right)^2
    \inf_{\test} \left \{ \P_{Q,P_0} \left(\test(\what{\theta}) \neq 0
    \right) +  \P_{Q,P_1} \left(\test(\what{\theta}) \neq 1 \right)
    \right\} \\
    & = \frac{r^2 \delta^{2 - 2 / \moment}}{2}
    \left(1 -
    \tvnorm{\marginprob_0^n - \marginprob_1^n}\right)  \\
    & \ge \frac{r^2 \delta^{2 - 2/\moment}}{2}
    \left(1 - 2 n \diffp \tvnorm{P_0 - P_1}\right)
    = \frac{r^2 \delta^{2 - 2/\moment}}{2}
    \left(1 - 2 n \diffp \delta\right),
  \end{align*}
  where we have used the contraction inequality of
  Theorem~\ref{theorem:contraction}.
  Choosing $\delta = 1 / (4 n \diffp)$, we substitute to find
  \begin{equation*}
    \minimax_n(\theta(\mc{P}_\moment^1(r)), Q, |\cdot|^2)
    \ge \frac{r^2\delta^{2 - 2 / \moment}}{4}
    = \frac{r^2}{4 \cdot 4^{2- 2 / \moment}} \left(\frac{1}{n \diffp}\right)^{
      \frac{2 \moment - 2}{\moment}}.
  \end{equation*}
  Our choice of $\channel$ was arbitrary, so
  once we note that the lower bound $r^2 /n$ on minimax
  estimation of a mean holds even in non-private settings, we obtain
  the lower bound.
\end{proof}

Inequality~\eqref{eqn:tv-lower-bound} exhibits some interesting effects of
privacy, even under such weak definitions as total variation privacy. We
might like to let $\diffp$ approach zero---meaning that the privacy
guarantees become stronger---as the sample size $n$ grows. If the
distribution is bounded, with $\norm{X}_2\leq r$ always, then taking
$\moment = \infty$ is possible and the lower bound in
\eqref{eqn:tv-lower-bound} scales as $r^2 / n + r^2 / (n^2 \diffp^2)$.  The
proposition then suggests (and we show later) that we can allow privacy at a
level of $\diffp = 1 / \sqrt{n}$ without negatively affecting convergence
rates.  Under the weaker assumption that $2<\moment < \infty$, however, the
proposition disallows such quickly decreasing $\diffp$; if $\diffp \ll
n^{\frac{2 - \moment}{2 \moment - 2}}$, there is a degradation in rate of
convergence. Moreover, if all we can guarantee is a second moment bound
($\moment = 2$), then any amount of privacy $\diffp < 1$ forces the rate to
degrade, and it is impossible to take $\diffp \to 0$ as $n \to \infty$
without suffering non-parametric rates of convergence.

\subsubsection{Support estimation under total variation privacy}

As our second example, we consider a support estimation problem, where the
goal is to find the support of a uniform distribution. In particular, we
would like to find $\theta \in \R_+$ when we know that $X \sim \uniform[0,
  \theta]$. In the private case, given a sample $X_1, \ldots, X_n \simiid
\uniform[0, \theta]$, a well-known estimator is to use the first order
statistic, $X_{(1)} = \max_i \{X_i\}$. This yields
$\E[(X_{(1)} - \theta)^2] \lesssim \theta^2 / n^2$ (see, for example,
the standard text of \citet{LehmannCa98}). Under even the weakest privacy
constraints, however, this rate is unattainable.
\begin{proposition}
  \label{proposition:uniform-lower-bound}
  Let $\tvchannels{\diffp}$ denote the family of $\diffp$-TV-private
  channels, and for $t > 0$ let $\mc{P}_t$ denote the collection of uniform
  distributions $\uniform[0, \theta]$ with $\theta \le t$. Then
  in absolute value error,
  \begin{equation}
    \label{eqn:uniform-lower-bound}
    \minimax_n(\theta(\mc{P}_t), \tvchannels{\diffp},
    |\cdot|) \ge \frac{1}{32} \frac{t}{n \diffp}.
  \end{equation}
\end{proposition}
Note that by Jensen's inequality, the lower
bound~\eqref{eqn:uniform-lower-bound} implies that
\begin{equation*}
  \sup_{\theta \le t} \frac{1}{t^2}
  \E_\theta[(\what{\theta} - \theta)^2] \ge
  \sup_{\theta \le t} \frac{1}{t^2}
  \E_\theta [|\what{\theta} - \theta|]^2
  \gtrsim \frac{1}{n^2 \diffp^2}.
\end{equation*}
There is thus no possible privacy setting allowing estimation at the
statistically efficient rate.

\begin{proof}
  Fix $\delta \in [0, t]$ and consider the two distributions $P_0 =
  \uniform[0, t-\delta]$ and $P_1 = \uniform[0, t]$. Comparing their
  variation distances, we have $\tvnorm{P_1 - P_2} = \delta / t$. Moreover,
  their respective maxima $\theta_0 = t-\delta$ and $\theta_1 = t$ satisfy
  the separation condition $|\theta_0 - \theta_1| = \delta$. Thus, by
  letting $\marginprob^n$ denote the marginal distribution of the released
  statistic, Le Cam's method (Lemma~\ref{lemma:le-cam}) coupled with the
  estimation-to-testing lower bound~\eqref{eqn:estimation-to-testing}
  implies
  \begin{equation*}
    \minimax_n(\theta(\mc{P}_t), \tvchannels{\diffp}, |\cdot|)
    \ge \frac{\delta}{4}
    \left(1 - \tvnorm{\marginprob_0^n - \marginprob_1^n}\right).
  \end{equation*}
  By the contraction inequality of Theorem~\ref{theorem:contraction},
  we obtain the lower bound
  \begin{align*}
    \E[|\what{\theta} - \theta|]
    & \ge \frac{\delta}{4}
    \left(1 - 2 n \diffp \tvnorm{P_0 - P_1}\right)
    = \frac{\delta}{4} \left(1 - \frac{2 n \diffp \delta}{t} \right).
  \end{align*}
  Choosing $\delta = t / (4 n \diffp)$ gives the
  result~\eqref{eqn:uniform-lower-bound}.
\end{proof}


\subsection{Lower bounds with variants of differential privacy}
\label{sec:diffp-lower-bounds}

We now turn to lower bounds on estimation when the mechanism $\channel$
satisfies (a variant of) differential privacy.  We will see that
this implies stronger lower bounds than those
implied by $\diffp$-total variation privacy, as we obtain results that
exhibit dependence on the ambient dimension $d$ as well as on the privacy
parameter $\diffp$. These lower bounds are based on a type of
``uniformity of probability mass'' argument.  Roughly, they
are consequences of a guarantee that differentially private estimators
$\what{\theta}$ assign relatively high probability mass to all parts of the
parameter space $\Theta$ as a consequence of the likelihood ratio guarantee
that is their definition.

As in the previous section, we have a (semi)metric $\metric$ on the
parameter space $\Theta$, and a family of distributions $\mc{P}$, where
$\packset$ indexes a subset $\{P_\packval\}_{\packval \in \packset} \subset
\mc{P}$.  Additionally, we assume there exists a distribution $P_0$ on the
space $\mc{X}$ such that for some (fixed) $p \in [0, 1]$, we have $(1 - p)
P_0 + p P_\packval \in \mc{P}$ for all $\packval\in\packset$.  With this
fixed $p$ in place, we may define the parameters we wish to estimate by
\begin{equation*}
  \theta_\packval \defeq \theta\big((1 - p) P_0 + p P_\packval\big),
\end{equation*}
where $\theta : \mc{P} \to \Theta$ is our population statistic. (We omit $p$
from our notation for $\theta_\packval$, leaving it implicit.) We then
define the separation of the set $\{\theta_\packval\}_{\packval \in
  \packset}$ by
\begin{equation}
  \label{eqn:separation}
  \separation{\packset}
  \defeq \min\left\{\metric(\theta_\packval, \theta_\altpackval)
  \mid \packval, \altpackval \in \packset, \packval \neq \altpackval\right\}.
\end{equation}

Now we come again to a standard testing problem: we choose a private
procedure $\what{\theta}$ (given by a channel $\channel$).  After we make
this choice, nature chooses one of the indices $\packval \in \packset$,
generating a sample $X_1, \ldots, X_n$ drawn i.i.d.\ from the distribution
$(1 - p) P_0 + p P_\packval$.  Our goal is then to estimate the parameter
$\theta_\packval=\theta\big((1 - p) P_0 + p P_\packval\big)$ accurately,
which (essentially) corresponds to identifying the index $\packval$ nature
chooses.  Under this setting, we can develop a result inspired by arguments
of \citet{HardtTa10} and \citet{BeimelBrKaNi13}.  In particular, we show
that private mechanisms necessarily are (non-trivially) likely to release
parameters far away from the true parameter. In our case, however, we study
population parameters rather than sample quantities (in contrast to
\citet{HardtTa10}), approximate privacy, and use a more classical estimation
framework rather than PAC learning~\cite{BeimelBrKaNi13}.

The following theorem (whose proof we give in
Section~\ref{sec:proof-population-packing}) is our main tool for proving
concrete lower lower bounds.
\begin{theorem}
  \label{theorem:population-packing}
  Fix $p \in [0, 1]$, and define $P_{\theta_\packval} = (1 - p) P_0 + p
  P_\packval \in \mc{P}$. Let $\what{\theta}$ be an $(\diffp,
  \delta)$-approximately differentially private estimator. Then
  \begin{equation*}
    \frac{1}{|\packset|}
    \sum_{\packval \in \packset}
    P_{\theta_\packval}\left(\metric(\what{\theta}, \theta_\packval)
    \ge \separation{\packset}\right)
    \ge
    \frac{(|\packset| - 1) \cdot \left(\half e^{-\diffp \ceil{np}}
      - \delta \frac{1 - e^{-\diffp \ceil{np}}}{1 - e^{-\diffp}}\right)}{
      1 + (|\packset| - 1) \cdot e^{-\diffp \ceil{np}}}.
  \end{equation*}
\end{theorem}

In the remainder of this section, we illustrate the consequences of this
result via two examples, the first on mean estimation and the second on
non-parametric density estimation. Roughly, we show that with appropriate
choice of the mixture parameter $p$,
Theorem~\ref{theorem:population-packing} implies it is difficult to
distinguish between the distributions $(1 - p) P_0 + p P_\packval$ and $(1 -
p) P_0 + p P_{\altpackval}$, when $\packval \neq \altpackval$, as long as
the packing set $\packset$ is large enough. In particular applications, we
show how this implies substantial dependence on the ambient dimension $d$ of
the parameter space.

\subsubsection{Mean estimation under (approximate) differential
  privacy}

For our first example, we provide a lower bound on convergence rates for
mean estimation problems.  We begin by recalling the
definition~\eqref{eqn:moment-distributions} of the set $\mc{P}_\moment^d(r)$
as distributions $P$ satisfying the non-central moment condition
$\E[\ltwo{X}^\moment] \le r^k$ (a similar result also holds under a central
moment condition).  We have the following result.
\begin{proposition}
  \label{proposition:diffp-mean-estimation}
  Let $\mc{Q}_{\diffp,\delta}$ denote the family of $(\diffp,
  \delta)$-approximately differentially private channels. Then
  for the mean estimation problem,
  \begin{equation}
    \label{eqn:moments-kill}
    \minimax_n\left(\theta(\mc{P}_\moment^d(r)), \mc{Q}_{\diffp,\delta},
    \ltwo{\cdot}^2\right)
    \gtrsim \frac{r^2}{n}
    + r^2 \min\bigg\{\bigg(\frac{d^2 \wedge \log^2 \frac{1}{\delta}}{
      n^2 \diffp^2}\bigg)^{\frac{\moment - 1}{\moment}}, 1 \bigg\}.
  \end{equation}
\end{proposition}

It is interesting to study the scaling---the relationships between dimension
$d$, privacy parameter $\diffp$, and sample size $n$---that
Proposition~\ref{proposition:diffp-mean-estimation} characterizes. Focusing
on the $\diffp$-differentially private case with \emph{bounded} random
variables (i.e.~$\moment = +\infty$) we see that the minimax lower bound
scales as $r^2 d^2 / (n^2 \diffp^2)$. In particular, the standard
(non-private) rate of convergence in squared $\ell_2$-error scales as $r^2 /
n$, as discussed preceding Proposition~\ref{proposition:minimax-lower-tv}.
This radius term $r$ implicitly encodes some dimension dependence (consider,
for example, a normal $\normal(0, \stddev^2 I_{d \times d})$ random
variable, which satisfies $\E[\ltwo{X}^2] = d \stddev^2$), so we see that
there is substantial \emph{additional} dimension dependence: to attain the
classical (non-private) rate of convergence, we must have $n$ scaling at
least as large as $n \ge d^2 / \diffp^2$. In general, for fixed $\diffp,
\delta$, we see that the lower bound in \eqref{eqn:moments-kill} scales as
$\frac{r^2}{n}$ if and only if
\begin{equation}
  \label{eqn:n-d-dep}
  n \ge \bigg(\frac{d^2 \wedge \log^2 \frac{1}{\delta}}{\diffp^2}\bigg)^{
    \frac{\moment - 1}{\moment - 2}}.
\end{equation}
We usually think of $\delta$ as decreasing quite quickly with $n$---as a
simple example, as $\delta = e^{-\sqrt{n}}$ with $\diffp \le 1$---so that
the sample complexity bound~\eqref{eqn:n-d-dep} implies the optimal
statistically efficient rate is possible only if $n \ge (d^2 /
\diffp^2)^{\frac{\moment - 1}{\moment - 2}}$.  Thus, at least for suitably
quickly decreasing $\delta$, we observe a quadratic-like penalty in
convergence rate from the dimension.

\begin{proof}[Proof of Proposition~\ref{proposition:diffp-mean-estimation}]
  Let $P_0$ be a point mass
  supported on $\{X = 0\}$, and for fixed $p \in [0, 1]$ let $P_\packval$
  be a point mass supported on $\{X = p^{-1/\moment} r \packval\}$, where
  $\ltwo{\packval} = 1$. For any such $p \in [0, 1]$, we claim that the
  mixture $P_{\theta_\packval} \defeq (1 - p) P_0 + p P_\packval$ lies in
  $\mc{P}_\moment^d(r)$. Indeed, we have $\theta(P_{\theta_\packval}) = \E_{P_{\theta_\packval}}[X] = p^{1 -
    1/\moment} r^k \packval$, and 
  \begin{align*}
    \E_{P_{\theta_\packval}}[\ltwo{X}^k]
    & = p \ltwo{p^{-1/\moment} r \packval}^\moment
    = p p^{-1} r^\moment
    = r^\moment.
  \end{align*}

  Applying a standard volume-based result on the construction of packing
  sets (e.g.~\cite{Ball97}), there exists a set $\packset\subset\R^d$ of
  cardinality at least $|\packset|\geq 2^d$, with $\norm{\packval}_2\leq1$
  for all $\packval\in\packset$ and with
  $\norm{\packval-\packval'}_2\geq\frac{1}{2}$ for all
  $\packval\neq\packval'\in\packset$.  Because $\theta(P_{\theta_\packval})
  = p^{1-1/\moment} r \packval$, we have separation
  \begin{equation*}
    \separation{\packset} \ge r
    p^{1-1/\moment} \min_{\packval \neq \altpackval} \ltwo{\packval -
      \altpackval} \ge (r/2) p^{1 - 1/\moment}.
  \end{equation*}

  Now, we apply the reduction of estimation to testing with this packing
  $\packset$,
  which implies
  \begin{align*}
    \sup_{P \in \mc{P}_\moment^d(r)}
    \E_P\left[\ltwos{\what{\theta} - \theta(P)}^2\right]
    & \ge \frac{1}{|\packset|} \sum_{\packval \in \packset}
    \E_{P_{\theta_\packval}}\left[
      \ltwos{\what{\theta} - \theta_\packval}^2\right]
    \ge \separation{\packset}^2
    \frac{1}{|\packset|} \sum_{\packval \in \packset}
    P_{\theta_\packval}\left(\metric(\what{\theta}, \theta_\packval) \ge
    \separation{\packset}\right).
  \end{align*}
  Using Theorem~\ref{theorem:population-packing},
  we thus obtain
  \begin{align*}
    \frac{1}{|\packset|} \sum_{\packval \in \packset}
    \E_{P_{\theta_\packval}}\left[\ltwos{\what{\theta} - \theta_\packval}^2
      \right]
    & \ge \left(\frac{r p^{1 - 1/\moment}}{2}\right)^2
    \frac{(|\packset| - 1) \cdot (\half e^{-\diffp \ceil{n p}} -
      \frac{\delta}{1 - e^{-\diffp}})}{
      1 + (|\packset| - 1) e^{-\diffp \ceil{np}}} \\
    & = \frac{r^2 p^{2 - 2/\moment}}{4} \cdot
    \frac{(2^d - 1) \cdot (\half e^{-\diffp \ceil{n p}} -
      \frac{\delta}{1 - e^{-\diffp}})}{
      1 + (2^d - 1) e^{-\diffp \ceil{np}}}.
  \end{align*}

  We now choose $p$ to (approximately) maximize the preceding display, which
  makes the average probability of error constant. Without loss of
  generality, we may assume that $d \ge 2$ (as
  Proposition~\ref{proposition:minimax-lower-tv} gives the result when $d =
  1$), so that $2^d - 1 \ge e^{d/2}$. We choose
  \begin{equation}
    \label{eqn:define-p}
    p = \frac{1}{n\diffp} \min\left\{
    \frac{d}{2} - \diffp,
    \log\left(\frac{1-e^{-\diffp}}{4\delta e^\diffp}\right)\right\}.
  \end{equation}
  The second term in the minimum~\eqref{eqn:define-p} is sufficiently
  small that
  \begin{equation*}
    \half e^{-\diffp \ceil{np}}
    - \frac{\delta}{1-e^{-\diffp}}
    \ge
    \frac{1}{4} e^{-\diffp (np + 1)} > 0,
  \end{equation*}
  and so we have
  \begin{equation*}
    \sup_{P \in \mc{P}_\moment^d(r)}
    \E_P\left[\ltwos{\what{\theta} - \theta(P)}^2\right]
    \ge
    \frac{r^2 p^{2 - 2/\moment}}{4}
    \cdot \frac{\frac{1}{4} e^{d/2} e^{-\diffp (np+1)}}{
      1+ e^{d/2} e^{-\diffp (np + 1)}}
    \ge \frac{r^2 p^{2 - 2/\moment}}{4} \cdot \frac{1}{8},
  \end{equation*}
  where we have used that the first term in the minimum~\eqref{eqn:define-p}
  implies that $e^{d/2} e^{-\diffp (np + 1)} \ge 1$.  For the
  result~\eqref{eqn:moments-kill}, substitute the value~\eqref{eqn:define-p}
  in the preceding display.
\end{proof}

\subsubsection{Nonparametric density estimation under differential privacy}

\newcommand{\pclass}{\mc{P}}
\newcommand{\numbin}{k}

We now turn to a non-parametric problem, showing how
Theorem~\ref{theorem:population-packing} provides lower bounds in this case
as well. Interestingly, we again obtain a result that suggests a penalty
from privacy that scales quadratically in (an analogue of) the dimension and
inversely in $n^2 \diffp^2$, as we saw in the preceding example. In this
case, however, the dimension is implictly chosen to make the problem
challenging. Formally, let $\pclass$ denote the
family of distributions supported on $[0, 1]^d$ with densities $f$ that are
$1$-Lipschitz continuous (with respect to the $\ell_2$-norm on $\R^d$).  We
assume we observe $X_1, \ldots, X_n \simiid P$, and we wish to estimate the
density $f$ from which the sample is drawn; we use $\ell_2^2$ error as our
loss, so that given an estimate $\what{f}$, we measure error via
$\ltwos{\what{f} - f}^2 = \int_0^1 (f(t) - \what{f}(t))^2 dt$.

In this case, we obtain the following result; we prove the result
only in the case of $\diffp$-differentially private channels ($\delta = 0$)
for simplicity. (See Section~\ref{sec:proof-density-lower-bound} for
the proof.)
\begin{proposition}
  \label{proposition:density-lower-bound}
  Let $\channels_\diffp$ denote the family of $\diffp$-differentially
  private channels. Then for a constant $c_d > 0$ that may depend on the
  dimension $d$,
  \begin{equation}
    \label{eqn:density-lower-bound}
    \minimax_n(\pclass, \channels_\diffp, \ltwo{\cdot}^2)
    \ge c_d \left[\frac{1}{n^{\frac{2}{2 + d}}}
    + \frac{1}{(n \diffp)^{\frac{2}{1 + d}}}\right].
  \end{equation}
\end{proposition}

The bound~\eqref{eqn:density-lower-bound} is matched by known
upper bounds. The $n^{-2/(2 + d)}$ term in the
bound is the well-known minimax rate for estimation of a Lipschitz density
on $[0, 1]^d$; a standard histogram estimator achieves this convergence rate
(see, for example, \citet{YangBa99} or \citet{Tsybakov09}). To attain the
latter part of the lower bound, we recall
\citet[Theorem 4.4]{WassermanZh10}. Making the immediate extension of their
results to $d$ dimensions, we note that \citeauthor{WassermanZh10} show that
constructing a standard histogram estimator with $\numbin$ equally sized
bins on $[0, 1]^d$, then adding independent Laplace noise (of appropriate
magnitude dependent on $\diffp$ and $\numbin$) to each of the bins, and
returning this histogram, gives an estimator $\what{f}_{\rm hist}$ that is
$\diffp$-differentially private and satisfies
\begin{equation}
  \E\left[\ltwos{\what{f}_{\rm hist} - f}^2\right]
  \lesssim \frac{\numbin}{n} + \frac{1}{\numbin^\frac{2}{d}}
  + \frac{\numbin^2}{n^2 \diffp^2}.
  \label{eqn:histogram-upper-bound}
\end{equation}
The first two terms are the standard bias-variance tradeoff in density
estimation (e.g.~\cite[Chapter 3]{Tsybakov09,DevroyeGy85}), while the last
$\numbin^2 / n^2 \diffp^2$ term is reminiscent of the
bounds~\eqref{eqn:moments-kill} in its additional quadratic
penalty. Choosing $\numbin = \min\{n^{1/(d+2)}, (n \diffp)^{1/(d+1)}\}$ in
expression~\eqref{eqn:histogram-upper-bound} gives the
bound~\eqref{eqn:density-lower-bound}.

We make one more remark on
Proposition~\ref{proposition:density-lower-bound}. Though our observations
$X_i$ are bounded, as are the densities we estimate, we may not take the
privacy parameter $\diffp$ to 0 as quickly as in the parametric
problems in the preceding section. Indeed, if $\diffp =
o(n^{-\frac{1}{2 + d}})$ as $n \to \infty$,
expression~\eqref{eqn:density-lower-bound} shows it is impossible to attain
the non-private rate. In contrast, in expression~\eqref{eqn:moments-kill},
we see that (assuming $\moment = +\infty$) as long as $\diffp
\gg n^{-1/2}$, as $n \to \infty$ we attain the classical
parametric rate.

\section{A few upper bounds for mean estimation}
\label{sec:upper-bounds}

\renewcommand{\truncate}[2]{\pi_{#2}\left({#1}\right)}
\renewcommand{\truncatebigg}[2]{\pi_{#2}\bigg({#1}\bigg)}

Having provided lower bounds on several (population) estimation problems, we
now focus on guarantees for convergence. We focus for concreteness on mean
estimation tasks under the moment
assumption~\eqref{eqn:moment-distributions}, as we wish to most simply
illustrate a few of the consequences of imposing privacy on estimators of
population quantities.
Our goal is thus to estimate the mean $\theta = \E[\statrv]$ of a
distribution $P \in \mc{P}_\moment^d(r)$, that is, distributions supported
on $\R^d$ satisfying the moment condition $\E[\ltwo{\statrv}^\moment] \le
r^\moment$ for some $k\geq 2$.

We first define our estimator, which is similar to estimators in some of our
earlier work~\cite{DuchiJoWa13_focs}. For $v \in \R^d$, let
$\truncate{v}{T}$ denote the projection of the vector $v$ onto the
$\ell_2$-ball of radius $T$. Now, let $W \in \R^d$ be a random vector (whose
distribution we specify presently); our private estimator is
\begin{equation}
  \label{eqn:truncated-mean-estimator}
  \what{\theta} \defeq \frac{1}{n} \sum_{i=1}^n \truncate{X_i}{T} + W.
\end{equation}
The estimator~\eqref{eqn:truncated-mean-estimator} is a type of robustified
estimator of location where outlying estimates are truncated to be within a
ball of radius $T$; similar ideas for estimation of parameters have been used
by Smith~\cite{Smith11} and are frequent in robust statistical
estimation~\cite{Huber81}.
By specific choices of $W$ and $T$, however, we can achieve order optimal
rates of convergence for our private estimators.

We consider three distributions for $W$
that variously satisfy our privacy definitions. 
Before giving them, we note that if we define $v = \frac{1}{n} \sum_{i=1}^n
\truncate{x_i}{T}$ and $v' = \frac{1}{n} \sum_{i=1}^n \truncate{x_i'}{T}$,
then it is clear that
\begin{equation}
  \label{eqn:truncated-means-close}
  \ltwo{v - v'}
  \le \frac{1}{n} \sum_{i=1}^n \ltwo{\truncate{x_i}{T}
    - \truncate{x_i'}{T}}
  \le \frac{1}{n}
  \sum_{i=1}^n \ltwo{x_i - x_i'} \wedge 2T
  \le \frac{2T}{n} \card\left\{i : x_i \neq x_i'\right\}.
\end{equation}
\paragraph{A mechanism for KL-divergence-based privacy}
\newcommand{\diffpkl}{\diffp_{\mathsf{KL}}} Let us first consider the
divergence-based variants of privacy, focusing on $\diffp$-KL
privacy~\eqref{eqn:kl-private}. In this case, take $W \sim \normal(0,
\frac{2 T^2}{n^2 \diffpkl} I_{d \times d})$.  Letting $\channel$ denote the
distribution of $\what{\theta}$, for samples $x_{1:n}$ and $x_{1:n}'$
differing in at most a single observation we have
\begin{align*}
  \dkl{\channelprob(\cdot \!\mid \statsample_{1:n})}{
    \channelprob(\cdot \!\mid \statsample_{1:n}')}
  & = \dkl{\normal\Big(v, \frac{2 T^2}{n^2 \diffpkl}\Big)}{
    \normal\Big(v', \frac{2 T^2}{n^2 \diffpkl}\Big)} \\
  & = \frac{n^2 \diffpkl}{4 T^2} \ltwo{v - v'}^2
  \le \frac{n^2 \diffpkl}{4T^2} \frac{4T^2}{n^2} = \diffpkl,
\end{align*}
because $\ltwo{v - v'} \le 2 T/n$. Therefore, this estimator achieves
KL-privacy as desired.

\paragraph{A mechanism for approximate differential privacy}
Turning now to the variants of differential privacy, we note that the
Hamming-Lipschitz guarantee~\eqref{eqn:truncated-means-close} implies that
if we take $W \sim \normal(0, \frac{2 T^2 \log \frac{1}{\delta}}{n^2
  \diffp^2} I_{d \times d})$, then the estimator $\what{\theta}$ is
$(\diffp, \delta)$-approximately differentially private (see, for example,
\citet{DworkKeMcMiNa06} or \citet[Section 1.3.2]{Hall13}).

\paragraph{A mechanism for smooth differential privacy}
Finally, we show how to satisfy the strongest variant of privacy, smooth
differential privacy (Definition~\ref{definition:smooth-privacy}).
In particular, using the metric $\privmetric(x, x') =
\ltwo{x - x'} \wedge 2T$ and $\dpriv(x_{1:n}, x_{1:n}')
= \frac{1}{2T} \sum_{i=1}^n \privmetric(x_i, x_i')$, we claim
that taking $W$ to have independent coordinates, each Laplace
distributed with density $p(w) \propto \exp(-\kappa|w|)$, where
$\kappa = \diffp n / 2 T \sqrt{d}$, satisfies smooth differential
privacy. Indeed, we have that the ratio of the densities
\begin{align*}
  \left|\log \frac{\channeldens(z \mid x_{1:n})}{
    \channeldens(z \mid x_{1:n}')}\right|
  & = \left|\frac{\diffp n}{2 T \sqrt{d}} \lone{v - z}
  - \frac{\diffp n}{2 T \sqrt{d}} \lone{v' - z}\right| \\
  & \le \frac{\diffp n}{2 T} \ltwo{v - v'}
  \le \frac{\diffp n}{n} \frac{1}{2 T} \sum_{i=1}^n \ltwo{x_i - x_i'}
  \wedge 2T
  = \diffp \cdot \dpriv(x_{1:n}, x_{1:n}'),
\end{align*}
where the final inequality uses the bound~\eqref{eqn:truncated-means-close}.
In particular, this additive Laplace noise mechanism
satisfies smooth differential privacy and, by extension, differential
privacy.

With these three mechanisms in place, we have the following proposition,
whose proof we provide in Section~\ref{sec:proof-minimax-upper}.
\begin{proposition}
  \label{proposition:minimax-upper}
  Consider the estimator~\eqref{eqn:truncated-mean-estimator}. The
  following hold.
  \begin{enumerate}[(i)]
  \item Choose $T = r (n^2 \diffpkl / d)^{1 / (2k)}$ and let
    $W \sim \normal(0, \frac{2 T^2}{n^2 \diffpkl} I_{d \times d})$. Then
    $\what{\theta}$ is $\diffpkl$-KL private, and
    \begin{equation*}
      \E[\ltwos{\what{\theta} - \E[X]}^2] \lesssim \frac{r^2}{n}
      + r^2 \left(\frac{d}{n^2 \diffpkl}\right)^{\frac{k - 1}{k}}.
    \end{equation*}
  \item Choose $T = r(n^2 \diffp^2 / (d \log \frac{1}{\delta}))^{1 / (2k)}$
    and let $W \sim \normal(0, \frac{2 T^2 \log \frac{1}{\delta}}{n^2
      \diffp^2} I_{d \times d})$. Then $\what{\theta}$ is $(\diffp,
    \delta)$-approximately differentially private, and
    \begin{equation*}
      \E[\ltwos{\what{\theta} - \E[X]}^2] \lesssim \frac{r^2}{n}
      +  r^2 \left(\frac{d \log \frac{1}{\delta}}{
        n^2 \diffp^2}\right)^{\frac{k - 1}{k}}.
    \end{equation*}
  \item Choose $T = r(n \diffp / d)^{1 / k}$ and let $W$ have independent
    $\laplace(\diffp n / (2 T \sqrt{d}))$-distributed coordinates.
    Then $\what{\theta}$ is $\diffp$-differentially private
    and $(\privmetric, \diffp)$-smoothly differentially private
    (Def.~\ref{definition:smooth-privacy}) with
    metric $\privmetric(x, x') = \ltwo{x - x'} \wedge 2 T$,
    and
    \begin{equation*}
      \E[\ltwos{\what{\theta} - \E[X]}^2] \lesssim \frac{r^2}{n}
      +  r^2 \left(\frac{d^2}{
        n^2 \diffp^2}\right)^{\frac{k - 1}{k}}.
    \end{equation*}
  \end{enumerate}
\end{proposition}

Proposition~\ref{proposition:minimax-upper} shows that many of the lower
bounds we have provided on population estimators in
Section~\ref{sec:lower-bounds} are tight. We summarize each of the
convergence guarantees in Table~\ref{table:upper-and-lower-bounds}, which
shows upper and lower bounds on estimation of a population mean that we have
derived. (Note that by Pinsker's inequality~\cite{CoverTh06}, $2\tvnorm{P_0
  - P_1}^2 \le \dkl{P_0}{P_1}$, so that lower bounds for $\diffp$-total
variation privacy imply lower bounds for $\sqrt{\diffpkl}$-KL privacy, and
convergence guarantees for $\diffpkl$-KL private estimators give convergence
guarantees for $\diffp^2$-TV private estimation.)  While our
bounds for $\diffpkl$-KL and $\diffp$-TV private estimators are not sharp---we
are missing a factor of the dimension $d$ between upper and lower
bounds---we see that divergence-based privacy allows substantially better
convergence guarantees as a function of the dimension as compared with
differential privacy. However, it does not permit better scaling with the
moments $\moment$ of the problem; all privacy guarantees suffer as the
number $\moment$ of moments available shrinks.  Moreover,
Proposition~\ref{proposition:minimax-upper}, when coupled with the lower
bounds provided by Proposition~\ref{proposition:diffp-mean-estimation},
shows that there is (essentially) no difference in estimation rates between
smooth differential privacy and differential privacy. In a sense, it is
possible to provide even stronger guarantees than differential privacy
without suffering in performance.

\begin{table}[t]
  \begin{center}
    \begin{tabular}{|c||c|c|}
      \hline
      Privacy type & Upper bound & Lower bound \\ \hline \hline
      No privacy constraint
      & $\displaystyle{\frac{1}{n}}$
      & $\displaystyle{ \frac{1}{n}}$ \\ \hline
      $\diffp$-differential (Def.~\ref{definition:differential-privacy})
      & $\displaystyle{\left(\frac{d^2}{n^2 \diffp^2}\right)^{
          \frac{\moment - 1}{\moment}} + \frac{1}{n}}$
      & $\displaystyle{\left(\frac{d^2}{n^2 \diffp^2}\right)^{
          \frac{\moment - 1}{\moment}} + \frac{1}{n}}$ \\ \hline
      $(\ltwo{\cdot}, \diffp)$-smooth (Def.~\ref{definition:smooth-privacy})
      & $\displaystyle{\left(\frac{d^2}{n^2 \diffp^2}\right)^{
          \frac{\moment - 1}{\moment}} + \frac{1}{n}}$
      & $\displaystyle{\left(\frac{d^2}{n^2 \diffp^2}\right)^{
          \frac{\moment - 1}{\moment}} + \frac{1}{n}}$ \\ \hline
      $(\diffp, \delta)$-approximate
      (Def.~\ref{definition:approximate-dp})
      & $\displaystyle{\left(\frac{d \log\frac{1}{\delta}}{n^2 \diffp^2}
        \right)^{\frac{\moment - 1}{\moment}} + \frac{1}{n}}$ &
      $\displaystyle{\left(\frac{d^2 \wedge \log^2\frac{1}{\delta}}{
          n^2 \diffp^2}\right)^{\frac{\moment - 1}{\moment}}
        + \frac{1}{n}}$ \\ \hline
      $\diffp$-TV (Def.~\ref{definition:divergence-privacy})
      & $\displaystyle{\left(\frac{d}{n^2 \diffp^2}\right)^{
          \frac{\moment - 1}{\moment}} + \frac{1}{n}}$
      & $\displaystyle{\left(\frac{1}{n^2 \diffp^2}\right)^{\frac{\moment - 1}{\moment}}
        + \frac{1}{n}}$  \\ \hline
    \end{tabular}
  \end{center}
  \caption{\label{table:upper-and-lower-bounds} Our known upper and lower
    bounds on the minimax risk for estimation of the mean of a distribution,
    given $n$ i.i.d. observations from a distribution $P$ on
     $\statdomain \subset \R^d$, where $P$
    satisfies the moment condition $\E[\ltwo{X}^\moment] \le 1$. The
    minimax risk is measured in \emph{squared} $\ell_2$-error.}
\end{table}

\section{Summary and open questions}
\label{sec:open-questions}

In this paper, we have provided a variety of definitions and formalisms for
privacy, as well as reviewing definitions already present in the literature.
We showed that testing-based definitions of privacy, which provide \emph{a
  priori} protection against disclosures of sensitive data, have some
similarities with differential privacy and related notions of privacy.  On
the other hand, differential privacy provides \emph{posterior} guarantees of
privacy and testing, and is in fact equivalent to variants of testing-based
notions of privacy that provide protection against inferences conditional on
the output of the private procedure.

To complement the definitional study we provide, we also investigated
consequences of our definitions for different estimation tasks for
population quantities. We identified a separation
between estimating means under (smooth) differential, approximate
differential, and the divergence-based ({\em a priori} testing) versions of
privacy, as exhibited by Table~\ref{table:upper-and-lower-bounds}.  It is
clear that there are many open questions remaining: first, our results are
not all sharp, as our upper and lower bounds match precisely only for the
strongest variants of privacy. Perhaps more interestingly, the weakest
(testing-based) definitions of total variation privacy is unsatisfactory
(recall the ``release-one-at-random'' scenario in
Example~\ref{ex:release-one}), but perhaps other divergences
(Definition~\ref{definition:divergence-privacy}) provide satisfactory
privacy protection. Such schemes allow substantially better estimation than
differential privacy constraints, as shown in
Table~\ref{table:upper-and-lower-bounds}, and may provide adequate
assurances of privacy in scenarios with a weaker adversary.

We believe that future work on alternate definitions of privacy, which
consider weaker adversaries (see \citet{BassilyGrKaSm13}), should be
fruitful. For example, differential privacy is equivalent to guarantees that
an adversary's posterior beliefs on the presence or absence of a data point
$x$ in a sample $X_{1:n}$ cannot be too different from his prior
beliefs---no matter the adversary's prior~\cite{KasiviswanathanSm13}. Can
restrictions on an adversary's prior beliefs, as studied by
\citet{BassilyGrKaSm13}, allow more accurate estimation? We believe any
proposal for privacy definitions should also include an exploration of the
fundamental limits of inferential procedures, as without such an
understanding, it is difficult to balance statistical utility and disclosure
risk. We hope that the techniques and insights we have developed here
provide groundwork for such future study into the tradeoffs between privacy
guarantees and estimation accuracy.


\subsection*{Acknowledgments}

We thank Philip Stark and Martin Wainwright for several insightful
conversations on and feedback about the paper, and Philip for suggesting
several variants of privacy and testing inequalities.

\appendix

\section{Proofs related to privacy definitions}

In this section, we collect proofs of the equivalence between our various
notions of privacy as well as a few consequences of our different
definitions.

\subsection{Proof of Proposition~\ref{proposition:diffp-test-errors}}
\label{sec:proof-diffp-test-errors}

We begin by proving that inequality~\eqref{eqn:hypothesis-test-dp} is
equivalent to $\diffp$-differential privacy.  Indeed, let $A \subset \Theta$
be an arbitrary set and let $\test(\theta)\defeq\indicb{\theta \in A}$.
Then if inequality~\eqref{eqn:hypothesis-test-dp} holds, we have
\begin{equation*}
  e^\diffp \channel(A \mid H_0) + (1 - \channel(A \mid H_1))
  \ge 1 - \delta, ~~~ \mbox{or} ~~~
  e^\diffp \channel(A \mid H_0) + \delta \ge \channel(A \mid H_1),
\end{equation*}
and similarly we have
\begin{equation*}
  (1 - \channel(A^c \mid H_0)) + e^\diffp \channel(A^c \mid H_1)
  \ge 1 - \delta, ~~~\mbox{so}  ~~~
  e^\diffp \channel(A^c \mid H_1) + \delta \ge \channel(A^c \mid H_0).
\end{equation*}
Since $A$ was arbitrary, the channel $\channel$ satisfies
Definition~\ref{definition:differential-privacy}. The other direction
is trivial.

Now we demonstrate inequality~\eqref{eqn:disclosure-risk}.
Applying \eqref{eqn:hypothesis-test-dp} twice, we have
\begin{equation*}
  \channel(\test = 1 \mid H_0) + e^\diffp\cdot \channel(\test = 0 \mid H_1)
  \ge 1 - \delta
  ~~~ \mbox{and} ~~~
  \channel(\test = 0 \mid H_1) + e^\diffp\cdot \channel(\test = 1 \mid H_0)
   \ge 1 - \delta
\end{equation*}
(where the second version holds by swapping $x_{1:n}$ with $x'_{1:n}$, and
replacing $\psi$ with $1-\psi$, then applying
\eqref{eqn:hypothesis-test-dp}).  Adding these two inequalities together, we
obtain
\begin{equation*}
  (e^\diffp + 1)\cdot
  \left(\channel(\test = 1 \mid H_0) + \channel(\test = 0 \mid H_1)\right)
  \ge 2 - \delta,
\end{equation*}
proving the first inequality in \eqref{eqn:disclosure-risk}. The second
statement of the inequality follows because $\frac{2}{1+e^\diffp}\geq
1-\frac{\diffp}{2}$ for all $\diffp\geq 0$.

\subsection{Proof of Observation~\ref{observation:information-processing}}
\label{sec:proof-information-processing}

The first statement of the observation
is immediate because of the data processing inequality
for $f$-divergences (see, e.g.~\citet[Theorem 14]{LieseVa06}):
we are guaranteed that for any samples $x_{1:n}$ and $x'_{1:n}$ in $\mc{X}^n$,
\begin{equation*}
  \df{\channel'(\cdot \!\mid x_{1:n})}{\channel'(\cdot \!\mid x'_{1:n})}
  \le \df{\channel(\cdot \!\mid x_{1:n})}{\channel(\cdot \!\mid x'_{1:n})}
\end{equation*}
by the Markovian construction of $\channel'$ from $\channel$, that is,
$\channel'=\channel''\circ\channel$ for some $\channel''$
by Definition~\ref{definition:deficient-channel}.

For the second observation, take two
samples $x_{1:n}$, $x'_{1:n}$ with $\dham(x_{1:n}, x'_{1:n}) \le 1$.  We use
the fact that $\E[W]=\int_0^\infty \P\{W\geq t\} dt$ for any non-negative
random variable $W$. We have that
$\channel' = \channel'' \circ \channel$ for some $\channel''$, so
\begin{align*}
  \channel'(A \mid x_{1:n})
  = \E_{Y\sim \channel(\cdot\mid x_{1:n})}[\channel''(A \mid Y)]
  & = \int_0^1 \P_{Y\sim \channel(\cdot\mid x_{1:n})}\left\{
  \channel''(A\mid Y)\geq t\right\} dt \\
  & = \int_0^1 \channel\left(\{y:\channel''(A\mid y)\geq t\}
  \mid x_{1:n}\right) dt.
\end{align*}
Applying the same reasoning to the sample $x'_{1:n}$, and using the fact
that $\channel$ is $(\diffp,\delta)$-differentially private, we then have
\begin{align*}
  \channel'(A \mid x'_{1:n})
  & = \int_0^1 \channel\left(
  \{y:\channel''(A\mid y)\geq t\}
  \mid x'_{1:n}\right) dt\\
  & \le \int_0^1 \left[ e^\diffp \channel\left(
    \{y:\channel''(A\mid y)\geq t\}\mid x_{1:n}\right) + \delta \right] dt 
  = e^\diffp \channel'(A\mid x_{1:n}) + \delta.
\end{align*}

\subsection{Proof of Theorem~\ref{theorem:chtp-implies-diffp}}
\label{sec:proof-chtp-implies-diffp}

We split the proof into the two statements: differential privacy
implies conditional hypothesis testing privacy, and conditional hypothesis
testing privacy (for
$\channel$ and for all less informative channels $\channel'\deficient
\channel$) implies differential privacy.

\subsubsection{Differential privacy implies conditional
  hypothesis testing privacy}

We need to show that for any samples $x_{1:n}$ and $x'_{1:n}$ differing in
at most one observation and measurable sets $A \subset \mc{Y}$ satisfying
$\channel(A \mid x_{1:n}) \wedge \channel(A \mid x'_{1:n}) \ge \deltach$,
\begin{equation*}
  \channel(\test = 1
  \mid x_{1:n}, Y\in A)
  + \channel(\test = 0 \mid x'_{1:n}, Y\in A)
  \ge
  1-\diffpch.
\end{equation*}
We assume that $\channel(A\mid x'_{1:n})\geq
(e^{2\diffp}-e^{\diffp})^{-1}\delta$, as otherwise CHTP is satisfied
regardless.

Let $B=\test^{-1}(\{1\})\subset \mc{Y}$ be the
acceptance region for the test $\test$.  Then by Bayes' rule and
differential privacy, we have
\begin{align*}
  \channel(\test = 1 \mid x_{1:n}, Y\in A)
  + \channel(\test = 0 \mid x'_{1:n}, Y\in A)
  & = \frac{\channel(A\cap B\mid x_{1:n})}{\channel(A\mid x_{1:n})}
  + \frac{\channel(A\cap B^c\mid x'_{1:n})}{\channel(A\mid x'_{1:n})} \\
  & \ge \frac{e^{-\diffp}( \channel(A\cap B\mid x'_{1:n}) - \delta)}{
    e^\diffp \channel(A\mid x'_{1:n}) + \delta}
  + \frac{\channel(A \cap B^c\mid x'_{1:n})}{\channel(A\mid x'_{1:n})} \\
  & \stackrel{(i)}{\ge}
  \frac{e^{-\diffp}( \channel(A\cap B\mid x'_{1:n}) - \delta)}{
    e^{2\diffp} \channel(A\mid x'_{1:n})}
  +\frac{\channel(A\cap B^c\mid x'_{1:n})}{\channel(A\mid x'_{1:n})},
\end{align*}
where inequality (i) follows from the assumption that $\channel(A
\mid x'_{1:n}) \ge \deltach = (e^{2\diffp} - e^\diffp)^{-1} \delta$.
Adding the fractions in the previous display, we obtain
\begin{align*}
  \lefteqn{\channel(\test = 1 \mid x_{1:n}, Y\in A)
    + \channel(\test = 0 \mid x'_{1:n}, Y\in A)} \\
  & \ge
  \frac{e^{-3\diffp} \cdot \channel(A \cap B\mid x'_{1:n})
    + \channel(A\cap B^c\mid x'_{1:n})}{\channel(A \mid x'_{1:n})}
  -\frac{e^{-3\diffp}\delta}{\channel(A \mid x'_{1:n})} \\
  & \geq e^{-3\diffp}\frac{\channel(A\cap B\mid x'_{1:n})
    +\channel(A\cap B^c\mid x'_{1:n})}{\channel(A\mid x'_{1:n})}
  -\frac{e^{-3\diffp}\delta}{ (e^{2\diffp}-e^{\diffp})^{-1}\delta}\\
  & = e^{-3\diffp}-(e^{2\diffp}-e^\diffp)e^{-3\diffp}
  = e^{-3\diffp} + e^{-2\diffp} - e^{-\diffp} = 1 - \diffpch,
\end{align*}
where the second inequality follows again by assumption that
$\channel(A \mid x'_{1:n}) \ge \deltach$.

\subsubsection{CHTP implies DP}

First, solving for $\diffpch$ and $\deltach$ in the statement of the theorem,
we have
\begin{equation*}
  \diffpch = \frac{e^{\frac{\diffp}{2}} - 1}{e^{\frac{\diffp}{2}} + 1}
  ~~~\mbox{and} ~~~
  \deltach = \delta\cdot e^{-\frac{\diffp}{2}}\;.
\end{equation*}
We need to show that
$\channel$ is $(\diffp,
\delta)$-differentially private, as long as for any $\wt{\channel}\deficient\channel$,
and
for any samples $x_{1:n}$ and $x'_{1:n}$ with
$\dham(x_{1:n}, x'_{1:n}) \le 1$, we have
\begin{multline}  \label{eqn:chpt-upper-bound}
  \wt{\channel}(\test = 1 \mid x_{1:n}, ; Z \in \wt{A})
  + \wt{\channel}(\test = 0 \mid x'_{1:n}; Z \in \wt{A}) \ge 1 - \diffpch\\
  \text{ \ for any set $\wt{A} \subset \mc{Z}$
with $\wt{\channel}(\wt{A} \mid x_{1:n}) \wedge
\wt{\channel}(\wt{A} \mid x'_{1:n}) \ge \deltach$}.
\end{multline}

For the sake of contradiction, let us assume that $\channel$ is 
not $(\diffp,\delta)$-differentially private, and so
there is a set $B$ and two samples $x_{1:n}$ and $x'_{1:n}$ with
$\dham(x_{1:n}, x'_{1:n}) \le 1$ such that
\begin{equation}
  \label{eqn:contradiction-diffp}
  \channel(B \mid x_{1:n}) > e^\diffp \channel(B \mid x'_{1:n}) + \delta.
\end{equation}
In particular, we will show that if there is a set $B$ satisfying
inequality~\eqref{eqn:contradiction-diffp}, then the upper
bound~\eqref{eqn:chpt-upper-bound} fails to hold. Let $C=B^c$
be the complement of $B$. 
 Set the thresholds
\begin{equation*}
  t_B \defeq 1 \wedge \frac{\channel(C \mid x'_{1:n})}{
    \channel(B \mid x_{1:n})}
  ~~~ \mbox{and} ~~~
  t_C \defeq 1 \wedge \frac{\channel(B \mid x_{1:n})}{\channel(C \mid x'_{1:n})}.
\end{equation*}

Let the channel $\wt{\channel}$ be defined by $\wt{\channel}(\cdot \!\mid X)
= \channel(\cdot \!\mid X) \times \uniform[0, 1]$, that is, the output of
$\wt{\channel}$ conditional on $X$ is the pair $(Y, U)$, where $Y \sim
\channel(\cdot \!\mid X)$ and $U$ is an independent uniform random
variable. In this case, we have the relation $\wt{\channel} \deficient
\channel$, so that $\wt{\channel}$ must satisfy
inequality~\eqref{eqn:chpt-upper-bound} for any test $\test$
 and samples $x_{1:n}$ and $x'_{1:n}$ satisfying
$\dham(x_{1:n},x'_{1:n}) \le 1$. If we define the Cartesian products
\begin{equation*}
  \wt{B} \defeq B \times [0, t_B]
  ~~~ \mbox{and} ~~~
  \wt{C} \defeq C \times [0, e^{-\frac{\diffp}{2}} t_C],
\end{equation*}
we also obtain the following pair of inequalities:
\begin{subequations}
  \label{eqn:wt-x01}
  \begin{equation}
    \begin{split}
      \wt{\channel}(\wt{B} \mid x_{1:n})
      = t_B \channel(B \mid x_{1:n})
      & \stackrel{\text{By }\eqref{eqn:contradiction-diffp}}{>}
      e^\diffp t_B \channel(B\mid x'_{1:n}) + t_B \delta
      = e^\diffp \wt{\channel}(\wt{B} \mid x'_{1:n}) + t_B\delta \\
      & ~~ \ge ~ e^\diffp \wt{\channel}(\wt{B} \mid x'_{1:n})
    \end{split}
    \label{eqn:wt-B-x0}
  \end{equation}
  and
  \begin{align}
    \wt{\channel}(\wt{C} \mid x'_{1:n})
    & = e^{-\frac{\diffp}{2}} t_C\channel(C\mid x'_{1:n})
    >
    e^{-\frac{\diffp}{2}} t_C \channel(C \mid x_{1:n})
    = \wt{\channel}(\wt{C} \mid x_{1:n}),
    \label{eqn:wt-C-x1}
  \end{align}
  where the strict inequality above follows
  from assumption~\eqref{eqn:contradiction-diffp}, as
  $\channel(C \mid
x_{1:n}) < \channel(C \mid x'_{1:n})$.
\end{subequations}
Moreover, we have the string of equalities
\begin{align}
  \wt{\channel}(\wt{B} \mid x_{1:n})
  = \left(1 \wedge \frac{\channel(C \mid x'_{1:n})}{
    \channel(B \mid x_{1:n})}\right) \channel(B \mid x_{1:n})
  & = \channel(C \mid x_{1:n}) \wedge \channel(C \mid x'_{1:n})
  \label{eqn:B-C-equal-channel} \\
  & = \left(1 \wedge \frac{\channel(B \mid x_{1:n})}{
    \channel(C \mid x'_{1:n})}\right) \channel(C\mid x'_{1:n})
  = e^{\frac{\diffp}{2}} \wt{\channel}(\wt{C} \mid x'_{1:n}). \nonumber 
\end{align}

With the strict inequalities~\eqref{eqn:wt-x01} and
equation~\eqref{eqn:B-C-equal-channel}, we can derive our desired
contradiction to the testing upper bound~\eqref{eqn:chpt-upper-bound},
which we prove by conditioning on $Z \in \wt{A}\coloneqq \wt{B}\cup\wt{C}$. 
First, we must check that $\wt{\channel}(\wt{A}\mid
x_{1:n}) \wedge \wt{\channel}(\wt{A} \mid x'_{1:n}) \ge \deltach$. Indeed,
since $\wt{A}=\wt{B}\cup\wt{C}$,
we have
\begin{equation*}
  \wt{\channel}(\wt{A} \mid x_{1:n})
  \wedge \wt{\channel}(\wt{A} \mid x'_{1:n})
  \ge \wt{\channel}(\wt{B} \mid x_{1:n}) \wedge \wt{\channel}(\wt{C}
  \mid x'_{1:n})
  = e^{-\frac{\diffp}{2}} \left(\channel(B \mid x_{1:n}) \wedge
  \channel(C \mid x'_{1:n})\right)
\end{equation*}
by Eq.~\eqref{eqn:B-C-equal-channel}. By
assumption~\eqref{eqn:contradiction-diffp}, we know that $\channel(B \mid
x_{1:n}) > \delta = e^{\frac{\diffp}{2}} \deltach$,
and inequality~\eqref{eqn:contradiction-diffp} implies
$1 - \channel(C \mid x_{1:n}) > e^\diffp(1 - \channel(C \mid x'_{1:n})) + \delta$,
and so
\begin{equation*}
  \channel(C \mid x'_{1:n}) > e^{-\diffp}(e^\diffp - 1) + e^{-\diffp} \delta
  + e^{-\diffp} \channel(C \mid x_{1:n}) \ge \delta
  = e^{\frac{\diffp}{2}} \deltach.
\end{equation*}
Therefore, the bound $\wt{\channel}(\wt{A}\mid
x_{1:n}) \wedge \wt{\channel}(\wt{A} \mid x'_{1:n}) \ge \deltach$ holds,
and we turn to contradicting the inequality~\eqref{eqn:chpt-upper-bound},
that is,
\begin{equation*}
  \wt{\channel}(\test = 1 \mid x_{1:n}; Z \in \wt{A})
  + \wt{\channel}(\test = 0 \mid x'_{1:n}; Z \in \wt{A}) \ge 1 - \diffpch.
\end{equation*}

To that end, we choose a particular test: let $\test(y)= \indicbs{y\in
  \wt{C}}$.  Then by Bayes' rule and the fact that $\wt{B}$ and $\wt{C}$ are
disjoint, we obtain
\begin{align*}
  \wt{\channel}(\test = 1 \mid x_{1:n}; Z \in \wt{A})
  & = \frac{\wt{\channel}(\wt{C} \mid x_{1:n})}{
    \wt{\channel}(\wt{B} \mid x_{1:n}) + \wt{\channel}(\wt{C} \mid x_{1:n})} \\
  & \stackrel{(i)}{<}
  \frac{\wt{\channel}(\wt{C} \mid x'_{1:n})}{
    \wt{\channel}(\wt{B} \mid x_{1:n}) + \wt{\channel}(\wt{C} \mid x'_{1:n})}
  \stackrel{(ii)}{=}
  \frac{\wt{\channel}(\wt{C} \mid x'_{1:n})}{
   e^{\frac{\diffp}{2}} \wt{\channel}(\wt{C} \mid x'_{1:n}) +  \wt{\channel}(\wt{C} \mid x'_{1:n})}
  = \frac{1}{e^{\frac{\diffp}{2}}+1}.
\end{align*}
where step (i) follows by inequality~\eqref{eqn:wt-C-x1} and
step (ii) follows from Eq.~\eqref{eqn:B-C-equal-channel}.
To lower bound the second probability in the testing upper
bound~\eqref{eqn:chpt-upper-bound}, we have
\begin{align*}
  \wt{\channel}(\test = 0 \mid x'_{1:n}; Z \in \wt{A})
  & = \frac{\wt{\channel}(\wt{B} \mid x'_{1:n})}{
    \wt{\channel}(\wt{B} \mid x'_{1:n}) + \wt{\channel}(\wt{C} \mid x'_{1:n})} \\
  & \stackrel{(i)}{<}
  \frac{e^{-\diffp}\wt{\channel}(\wt{B} \mid x_{1:n})}{
    e^{-\diffp} \wt{\channel}(\wt{B} \mid x_{1:n})
    + \wt{\channel}(\wt{C} \mid x'_{1:n})}
  \stackrel{(ii)}{=}
  \frac{e^{-\frac{\diffp}{2}} wt{\channel}(\wt{C} \mid x'_{1:n})}{
    e^{-\frac{\diffp}{2}} \wt{\channel}(\wt{C} \mid x'_{1:n})
    + \wt{\channel}(\wt{C} \mid x'_{1:n})}
  = \frac{1}{e^{\frac{\diffp}{2}}+ 1},
\end{align*}
where we have used inequality~\eqref{eqn:wt-B-x0} for step (i)
and Eq.~\eqref{eqn:B-C-equal-channel} again for step (ii).
Combining the two preceding displays,
we obtain
\begin{align*}
  \wt{\channel}(\test = 0 \mid x_{1:n}; Z \in \wt{A})
  +
  \wt{\channel}(\test = 1 \mid x'_{1:n}; Z \in \wt{A})    
  < \frac{2}{e^{\frac{\diffp}{2}}+1}
  & = 1 - \frac{e^{\frac{\diffp}{2}}-1}{e^{\frac{\diffp}{2}}+1}
  =  1-\diffpch,
\end{align*}
where we have recalled the definition of $\diffpch$. This
contradicts the testing bound~\eqref{eqn:chpt-upper-bound}.

\section{Proofs of Minimax Lower Bounds}
\label{appendix:lower-bounds}

In this section, we collect proofs of each of our minimax lower bounds
and their related results.

\subsection{Proof of Theorem~\ref{theorem:contraction}}
\label{sec:proof-contraction}

In this section we prove a slightly more general form of
Theorem~\ref{theorem:contraction}.  Let $\statprob_{0,i}$ and
$\statprob_{1,i}$, $i = 1, \ldots, n$ be probability distributions on
$\statdomain$, and let $\statprob_\packval^n$ be their $n$-fold products for
$\packval=0,1$ (that is, we draw independent, but not necessarily
identically distributed, observations $X_1\sim \statprob_{\packval,1}$,
\dots, $X_n\sim\statprob_{\packval,n}$).  Under $\diffp$-total-variation
privacy (Def.~\ref{definition:tv-private}), we will prove that
\begin{equation}\label{eqn:prop_product}
  \tvnorm{\marginprob_0^n - \marginprob_1^n}
  \le 2 \diffp \sum_{i=1}^n \tvnorm{\statprob_{0,i} - \statprob_{1,i}}.
\end{equation}
For the special case that $\statprob_{\packval,i}=\statprob_{\packval}$ for
all $i=1,\dots,n$ (for each $\packval=0,1$), this proves that
\begin{equation*}
  \tvnorm{\marginprob_0^n -
    \marginprob_1^n} \le 2 n \diffp \tvnorm{\statprob_0 - \statprob_1}.
\end{equation*}
The inequality $\tvnorm{\marginprob_0^n - \marginprob_1^n}
\le\tvnorm{\statprob_0^n - \statprob_1^n}$ is immediate from the classical data
processing inequality (cf.~\cite[Theorem 14]{LieseVa06}), so proving
inequality~\eqref{eqn:prop_product} is sufficient to prove the theorem.

Now we turn to the proof of \eqref{eqn:prop_product}.
By the product nature of $\statprob_\packval^n(\statsample_{1:n})$ for each 
$\packval=0,1$,
we have
\begin{equation*}
  d\statprob_0^n(\statsample_{1:n})
  - d\statprob_1^n(\statsample_{1:n})
  = \sum_{i=1}^n d \statprob_1^{i-1}(\statsample_{1:i-1})
  (d \statprob_{0,i}(\statsample_i) - d\statprob_{1,i}(\statsample_i))
  d \statprob_0^{n-i}(\statsample_{i+1:n}).
\end{equation*}
For any set $A \in \sigma(\channeldomain)$, we thus have
\begin{align*}
  \lefteqn{|\marginprob_0^n(A) - \marginprob_1^n(A)|
    = \bigg| \int_{\statdomain^n}
    \channel(A \mid \statsample_{1:n})
    d \statprob_0^n(\statsample_{1:n}) - d\statprob_1^n(\statsample_{1:n})
    \bigg|} \\
  & = \bigg|\sum_{i = 1}^n \int_{\statdomain^n}
  \channel(A \mid \statsample_{1:n})
  d \statprob_1^{i-1}(\statsample_{1:i-1})
  (d \statprob_{0,i}(\statsample_i) - d\statprob_{1,i}(\statsample_i))
  d \statprob_0^{n-i}(\statsample_{i+1:n})
  \bigg|, \\
  \intertext{and writing $x_{\setminus i} \defeq \{x_1,
\ldots, x_{i-1}, x_{i+1}, \ldots, x_n\}$,}
  & \le \sum_{i=1}^n
  \bigg|\int_{\statdomain^n} \left(\channel(A \mid \statsample_{1:n})
  - \channel(A \mid \statsample_{\setminus i}, \statsample_i')\right)
  d \statprob_1^{i-1}(\statsample_{1:i-1})
  (d \statprob_{0,i}(\statsample_i) - d\statprob_{1,i}(\statsample_i))
  d \statprob_0^{n-i}(\statsample_{i+1:n})
  \bigg| \\
  & \le \sum_{i=1}^n \sup_{\statsample_{1:n}, \statsample_{1:n}'}
  |\channel(A \mid \statsample_{1:n}) - \channel(A \mid \statsample_{1:n}')|
  \int_{\statdomain^n}   d \statprob_1^{i-1}(\statsample_{1:i-1})
  |d \statprob_{0,i}(\statsample_i) - d\statprob_{1,i}(\statsample_i)|
  d \statprob_0^{n-i}(\statsample_{i+1:n}),
\end{align*}
where the supremum is taken over samples with $\dham(\statsample_{1:n},
\statsample_{1:n}') \le 1$. By our privacy assumption, we have
\begin{equation*}
  \sup_{\statsample_{1:n}, \statsample_{1:n}'}
  |\channel(A \mid \statsample_{1:n}) - \channel(A \mid \statsample_{1:n}')|
  \le \sup_{\statsample_{1:n}, \statsample_{1:n}'}
  \tvnorm{\channel(\cdot \!\mid \statsample_{1:n}) -
    \channel(\cdot \!\mid \statsample_{1:n}')}
  \le \diffp,
\end{equation*}
and since $\int |d\statprob_{0,i} - d\statprob_{1,i}| = 2
\tvnorm{\statprob_{0, i} - \statprob_{1,i}}$, this completes the proof.


\subsection{Proof of Theorem~\ref{theorem:population-packing}}
\label{sec:proof-population-packing}

\newcommand{\psucc}{P_{\mathsf{succ}}}

We begin the proof of Theorem~\ref{theorem:population-packing} by stating a
lemma that shows, roughly, that a set $A$ with high probability under a
distribution $P_{\theta_\packval}$ must also have high probability under
$P_{\theta_{\altpackval}}$, so long as the estimator $\what{\theta}$ is
$\diffp$-differentially private.  We recall the definition of
$P_{\theta_\packval} = (1 - p) P_0 + p P_{\packval}$ (where the sample size
$n$ is implicit).
\begin{lemma}
  \label{lemma:mass-everywhere}
  Let $A$ be a measurable set, 
  and $\packval, \altpackval \in \packset$.
  Assume that $P_{\theta_\packval} \in \mc{P}$ for all $\packval$.
  Then
  if $\what{\theta}$ is $(\diffp, \delta)$-approximately
  differentially private,
  \begin{equation}
    P_{\theta_\packval}(\what{\theta} \in A)
    \ge e^{-\diffp \lceil np\rceil}
    \left[P_{\theta_\altpackval}(\what{\theta} \in A)
      - \half\right]
    - \delta \frac{1 - e^{-\diffp \ceil{np}}}{1 - e^{-\diffp}}.
    \label{eqn:bernoullis-approximate-diffp}
  \end{equation}
\end{lemma}

Deferring the proof of Lemma~\ref{lemma:mass-everywhere}, we now show how it
leads to a short proof of Theorem~\ref{theorem:population-packing}.  Let
$\ball_\epsilon(\theta) = \{\theta' \in \Theta : \metric(\theta, \theta')
\le \epsilon\}$. Then by assumption on $\separation{\packset}$, the balls
$\ball_{\separation{\packset}}(\theta_\packval)$ are disjoint for all $\packval$.
Now, for an estimator $\what{\theta}$, let
the average probability of success be
\begin{equation}
  \label{eqn:avg-success-prob}
  \psucc \defeq \frac{1}{|\packset|}
  \sum_{\packval \in \packset} P_{\theta_\packval}\left(\what{\theta}
  \in \ball_{\separation{\packset}}(\theta_\packval)\right).
\end{equation}
Then we have
\begin{align*}
  \psucc  = 1 - \frac{1}{|\packset|}
  \sum_{\packval \in \packset} P_{\theta_\packval}\left(\what{\theta}
  \not\in \ball_{\separation{\packset}}(\theta_\packval)\right)
  & \le
  1 - \frac{1}{|\packset|} \sum_{\packval \in \packset}
  P_{\theta_\packval}\bigg(\what{\theta} \in 
  \bigcup_{\altpackval \in \packset, \altpackval \neq \packval}
  \ball_{\separation{\packset}}(\theta_\altpackval)\bigg) \\
  & = 1 - \frac{1}{|\packset|}
  \sum_{\packval \in \packset}
  \sum_{\altpackval \in \packset, \altpackval \neq \packval}
  P_{\theta_\packval}\left(\what{\theta} \in
  \ball_{\separation{\packset}}(\theta_\altpackval)\right),
\end{align*}
where the inequality follows from the disjointness of the balls
$\ball_{\separation{\packset}}(\theta_\packval)$.  Using
Lemma~\ref{lemma:mass-everywhere}, we can lower bound the probability
$P_{\theta_\packval}(\what{\theta} \in
\ball_{\separation{\packset}}(\theta_\altpackval))$, whence we find that
\begin{align*}
  \psucc & \stackrel{\eqref{eqn:bernoullis-approximate-diffp}}{\le}
  1 - \frac{1}{|\packset|}
  \sum_{\packval \in \packset} \sum_{\altpackval \in \packset,
    \altpackval \neq \packval}
  \left[e^{-\diffp \ceil{n p}} \left(P_{\theta_\altpackval}
    \left(\what{\theta} \in \ball_{\separation{\packset}}(\theta_\altpackval)\right)
    - \half \right) - \delta \frac{1 - e^{-\diffp \ceil{np}}}{
      1 - e^{-\diffp}}\right] \\
  & = 1 - e^{-\diffp \ceil{np}}
  \frac{(|\packset| - 1)}{|\packset|} \sum_{\packval \in \packset}
  P_{\theta_\packval}
  \left(\what{\theta} \in \ball_{\separation{\packset}}(\theta_\packval)\right)
  + e^{-\diffp \ceil{n p}} \frac{|\packset| - 1}{2}
  + (|\packset| - 1) \delta \frac{1 - e^{-\diffp \ceil{np}}}{
    1 - e^{-\diffp}} \\
  & = 1 + (|\packset| - 1) \left[\frac{e^{-\diffp \ceil{np}}}{2}
    + \delta \frac{1 - e^{-\diffp \ceil{np}}}{1 - e^{-\diffp}}\right]
  - e^{-\diffp \ceil{np}} (|\packset| - 1) \psucc,
\end{align*}
where we have used the definition~\eqref{eqn:avg-success-prob} of $\psucc$.
Rearranging terms, we obtain
\begin{equation*}
  \psucc \le
  \frac{1 + (|\packset| - 1)\cdot \left(\half e^{-\diffp \ceil{np}}
    + \delta \frac{1 - e^{-\diffp\ceil{np}}}{1 - e^{-\diffp}}\right)}{
    1 + (|\packset|-1) \cdot e^{-\diffp \ceil{np}}}.
\end{equation*}
Lower bounding $1 - \psucc$ gives the theorem.

\begin{proof-of-lemma}[\ref{lemma:mass-everywhere}]
  Let $B = \{B_i\}_{i=1}^n$ be sequence of i.i.d.\ $\bernoulli(p)$ random
  variables.  Now, assume that observations are generated according to the
  following distribution: first, draw $W_1^0, \dots, W_n^0 \simiid P_0$ and
  draw $W^\packval_1,\dots,W^\packval_n \simiid P_\packval$.  Then for each
  $i$, if $B_i=0$, set $X_i = W_i^0$, while if $B_i = 1$, set
  $X_i=W^\packval_i$. By inspection, we have that observations are
  marginally drawn i.i.d.\ according to the mixture $P_{\theta_\packval} =
  (1 - p) P_0 + p P_\packval$.  Additionally, for fixed
  $\altpackval \in \packset$, generate an alternate
  sample by drawing
  $W^{\altpackval}_i \simiid P_{\altpackval}$ and setting
  \begin{equation*}
    X'_i = W_i^0 \cdot (1 - B_i) + W^\altpackval_i \cdot B_i
  \end{equation*}
  for each $i$. By construction, we observe that
  \begin{equation*}
    \dham(X_{1:n},X_{1:n}') \le B^\top \onevec.
  \end{equation*}

  By definition of $(\diffp, \delta)$-approximate differential
  privacy, we have for any fixed sequence $b \in \{0, 1\}^n$ that
  \begin{align}
    & \channel(\what{\theta} \in A \mid X_i = W_i^0\cdot (1 - b_i)
    + W^\packval_i\cdot b_i ~ \mbox{for~each~} i \in [n]) \nonumber \\
    &\geq e^{-\diffp b^\top \onevec}
    \channel(\what{\theta} \in A \mid X_i = W_i^0 \cdot (1 - b_i) 
    + W^{\altpackval}_i\cdot b_i ~ \mbox{for~each~} i \in [n])
    -\sum_{i=0}^{b^\top \onevec - 1} \delta e^{-\diffp i} \nonumber \\
    & = e^{-\diffp b^\top \onevec} 
    \channel(\what{\theta} \in A \mid X_i= W_i^0\cdot (1-b_i) +
    W^{\altpackval}_i\cdot b_i ~\mbox{for~each~} i \in [n])
    - \delta \frac{1 - e^{-\diffp b^\top \onevec}}{1-e^{-\diffp}}.
    \label{eqn:q-b-lower-bound}
  \end{align}
  By construction, we have
  \begin{align*}
    & P_{\theta_\packval}(\what{\theta} \in A)
    = \sum_{b \in \{0, 1\}^n} P(B = b) P_{\theta_\packval}(\what{\theta} \in A
    \mid B = b) \\
    & = \sum_{b \in \{0, 1\}^n} P(B = b)
    \int \channel(\what{\theta} \in A \mid X_i= w_i^0 \cdot (1-b_i)
    + w^\packval_i \cdot b_i \mbox{~for~} i \in [n])
    dP_0^n(w_{1:n}^0)dP_\packval^n(w^\packval_{1:n})
  \end{align*}
  Removing some terms in the summation corresponding to $b^\top \onevec \le
  \ceil{np}$ and integrating over the additional variables $w_i^\altpackval$,
  we obtain
  \begin{align*}
    & P_{\theta_\packval}(\what{\theta} \in A) \\
    & \geq\sum_{\substack{b \in \{0, 1\}^n\\ b^\top \onevec \leq \ceil{np}}}
    \!\!\!\!\! P(B = b)
    \int \channel(\what{\theta} \in A \mid X_i=w_i^0\cdot (1 - b_i)
    + w^\packval_i \cdot b_i \mbox{~for~} i \in [n])
    dP_0^n(w_{1:n}^0)dP_\packval^n(w^\packval_{1:n}) \\
    & = \sum_{\substack{b \in \{0, 1\}^n\\ b^\top \onevec \leq \ceil{np}}}
    \!\!\!\!\! P(B = b)
    \int \channel(\what{\theta} \in A \mid X_i=w_i^0\cdot (1-b_i) +
    w^\packval_i\cdot b_i \mbox{~for~} i \in [n]) dP_0^n(w^0_{1:n})
    dP_\packval^n(w^\packval_{1:n}) dP_{\altpackval}^n(w^{\altpackval}_{1:n}).
  \end{align*}
  Applying the approximate differential privacy lower
  bound~\eqref{eqn:q-b-lower-bound}, we obtain the further lower bound
  \begin{align*}
    P_{\theta_\packval}(\what{\theta} \in A)
    & \ge \sum_{\substack{b \in \{0, 1\}^n\\ b^\top \onevec \leq \ceil{np}}}
    P(B = b)
    \int
    \bigg[e^{-\diffp \ceil{np}}
      \channel\left(\what{\theta} \in A \mid X_i=w_i^0\cdot (1-b_i)
      + w^{\altpackval}_i\cdot b_i \mbox{~for~} i \in [n]\right)
      \ldots \\
      & \qquad\qquad \qquad\qquad\qquad\qquad ~
      - \delta \frac{1 - e^{-\diffp \ceil{np}}}{1-e^{-\diffp}}\bigg]
    dP_0^n(w^0_{1:n}) dP_\packval^n(w^\packval_{1:n})
    dP_{\altpackval}^n(w^{\altpackval}_{1:n}) \\
    & = \sum_{\substack{b \in \{0, 1\}^n\\ b^\top \onevec \leq 
        \ceil{np}}} P(B = b)
    \left(e^{-\diffp\ceil{np}}
    P_{\theta_{\altpackval}}(\what{\theta}\in A\mid B=b) -
    \delta \frac{1 - e^{-\diffp \ceil{np}}}{1-e^{-\diffp}}\right) \\
    & \ge e^{-\diffp\ceil{np}}
    P_{\theta_{\altpackval}}(\what{\theta}\in A, B^\top \onevec \leq \ceil{np})
    - \delta 
    \frac{1 - e^{-\diffp \ceil{np}}}{1-e^{-\diffp}}\\  
    &\geq e^{-\diffp\ceil{np}}
    \left( P_{\theta_{\altpackval}}(\what{\theta}\in A)
    - P( B^\top \onevec > \ceil{np})\right)
    - \delta \frac{1 - e^{-\diffp \ceil{np}}}{1-e^{-\diffp}},
  \end{align*}
  the last inequality following from a union bound.
  The median of the $\binomial(n,p)$ distribution is no larger
  than $\ceil{np}$, so we obtain
  \begin{align*}
    P_{\theta_\packval}(\what{\theta} \in A)
    \geq e^{-\diffp\ceil{np}}\left(
    P_{\theta_{\altpackval}}(\what{\theta}\in A) - \half\right)
    - \delta \frac{1 - e^{-\diffp \ceil{np}}}{1-e^{-\diffp}}.
  \end{align*}
  for any set $A$.
\end{proof-of-lemma}

\subsection{Proof of Proposition~\ref{proposition:density-lower-bound}}
\label{sec:proof-density-lower-bound}

The first term in the bound~\eqref{eqn:density-lower-bound} is a standard
result in nonparametric density estimation; see, for example, \citet[Theorem
  2.8]{Tsybakov09}, \citet[Chapter 4]{DevroyeGy85}, or \citet[Section
  6]{YangBa99}. We thus focus on the second term in the lower
bound~\eqref{eqn:density-lower-bound}.

Let $P_0$ be the uniform distribution on $[0, 1]^d$, with density $f \equiv
1$. Standard results in approximation theory and density estimation (see,
for example, the \citet[Chapter 4]{DevroyeGy85},
\citet{YangBa99}, or \citet[Section 5]{Lorentz66}) show the following
result: the packing entropy for the collection of $1$-Lipschitz
densities on $[0, 1]^d$ scales as $(1 / \epsilon)^d$.  More concretely,
there exist constants $c_0, c_1 > 0$ (that may depend on the dimension $d$)
such that for any $\epsilon \in \openleft{0}{1}$, there exists a collection
$\{f_\packval\}_{\packval \in \packset}$ of densities $f_\packval$, where
each density $f_\packval$ is $1$-Lipschitz continuous, $\ltwo{f_\packval -
  f_{\altpackval}} \ge c_0 \epsilon$, the set $\packset$ has cardinality
\begin{equation}
  \label{eqn:loads-of-densities}
  \log |\packset| \ge c_1 \frac{1}{\epsilon^d},
  ~~~ \mbox{and} ~~~
  f_\packval(x) \in [1 - \epsilon, 1 + \epsilon]
  ~ \mbox{for~} x \in [0, 1]^d.
\end{equation}

Now, choose $p \in \openleft{0}{1}$, and set $\epsilon = p$ in the
construction leading to the inequalities~\eqref{eqn:loads-of-densities}.
Then the density $1 + (1/p) (f_\packval - 1)$ is a valid density and is
$(1/p)$-Lipschitz. If $P_\packval$ denotes the distribution with this
density, then we have $(1 - p) P_0 + p P_\packval \in \pclass$,
and moreover, the mixture $(1 - p) P_0 + p P_\packval$
has density $(1 - p) + p[\frac{1}{p}(f_\packval - 1) + 1] = f_\packval$.
In particular, we have the separation (for the metric $\metric(f, g)
= \ltwo{f - g}$)
\begin{equation*}
  \separation{\packset} \ge c_0 p.
\end{equation*}
We now apply Theorem~\ref{theorem:population-packing}, setting $\delta=0$
as we are working with differential privacy.
Letting $f_P$ denote the density associated with the
distribution $P$, we obtain that for
any $\diffp$-differentially private estimator $\what{f}$ based
on $n$ observations and any $p \in [0, 1]$ that
\begin{align*}
  \sup_{P \in \pclass}
  \E\left[\ltwos{\what{f} - f_P}^2\right] & \ge
  \separation{\packset}^2 \left(\frac{\half
    \left(e^{c_1 (\frac{1}{p})^d} - 1\right)
    e^{-\diffp \ceil{np}}}{1 +
    (e^{c_1 (\frac{1}{p})^d} - 1) e^{-\diffp \ceil{np}}}
  \right) \\
  & \ge c_0^2 \, p^2
  \cdot \frac{e^{c_1 (\frac{1}{p})^d - \diffp \ceil{np}} -
    e^{-\diffp \ceil{np}}}{
    1 + e^{c_1 (\frac{1}{p})^d - \diffp \ceil{np}} -
    e^{-\diffp \ceil{np}}}.
\end{align*}
By choosing $p = \min\{\half (n \diffp/c_1)^{-\frac{1}{d+1}}, 1\}$ we obtain
\begin{equation*}
  c_0^2 \, p^2
  \frac{e^{c_1 (\frac{1}{p})^d - \diffp \ceil{np}} -
    e^{-\diffp \ceil{np}}}{
    1 + e^{c_1 (\frac{1}{p})^d - \diffp \ceil{np}} -
    e^{-\diffp \ceil{np}}}
  \ge 
  \frac{1}{(n \diffp)^{\frac{2}{d + 1}}}
  \cdot c_0^2
  \cdot c_d,
\end{equation*}
where $c_d$ is a constant that may depend on $d$. This gives the desired
result~\eqref{eqn:density-lower-bound}.

\section{Proof of Proposition~\ref{proposition:minimax-upper}}
\label{sec:proof-minimax-upper}

We begin our proof by presenting two lemmas, the first of which gives
a bound on the bias of our estimator, the second showing that the variance
of random vectors projected onto convex sets is always smaller than
the initial variance.

\begin{lemma}
  \label{lemma:proj-mean}
  Let $X$ be an arbitrary random vector with $\E[\ltwo{X}^\moment] < \infty$
  for some $\moment > 1$.
  Then there exists a vector $v \in \R^d$ satisfying $\ltwo{v} \le
  \E[\ltwo{X}^\moment] / ((\moment-1) T^{\moment-1})$
  such that $\E[\truncate{X}{T}] = \E[X] + v$.
\end{lemma}
\begin{proof}
  It is clear that $\E[\truncate{X}{T}] = \E[X] + \E[\truncate{X}{T} - X]$, and
  $\ltwo{\E[\truncate{X}{T} - X]} \le \E[\ltwo{\truncate{X}{T} - X}]$.
  But $\ltwo{X - \truncate{X}{T}} \le \ltwo{X} \cdot\indicb{\ltwo{X} > T}$, and
  consequently for any $\moment > 1$ we have
  \begin{equation*}
    \E[\ltwo{\truncate{X}{T} - X}]
    \le \E[\ltwo{X} \cdot\indicb{\ltwo{X} \ge T}]
    = \int_T^\infty \P(\ltwo{X} \ge t) dt
    \le \int_T^\infty \frac{\E[\ltwo{X}^\moment]}{t^\moment} dt
    = \frac{\E[\ltwo{X}^\moment]}{(\moment - 1) T^{\moment - 1}},
  \end{equation*}
  where the second inequality follows from Markov's inequality.
\end{proof}

\begin{lemma}
  \label{lemma:proj-variance}
  Let $X$ be an $\R^d$-valued random variable
  and let $\Pi_C(x)$ denote the (Euclidean) projection
  of a point $x \in \R^d$ onto a closed convex set $C$. Then
  $\E[\ltwos{\Pi_C(X)-\E[\Pi_C(X)]}^2] \le \E[\ltwos{X-\E[X]}^2]$.
\end{lemma}
\begin{proof}
  Let $X'$ be an i.i.d.~ copy of $X$. Then, since $\E[X]=\E[X']$, we have
  \begin{align*}
    \E[\ltwo{X-X'}^2]
    &=\E[\ltwo{X-\E[X] + \E[X'] - X'}^2]\\
    &=\E[\ltwo{X-\E[X]}^2] + \E[\ltwo{X'-\E[X']}^2]
    - 2\E[\< X-\E[X],X'-\E[X'] \>]\\ 
    &=\E[\ltwo{X-\E[X]}^2] + \E[\ltwo{X'-\E[X']}^2]\\
    &=2\E[\ltwo{X-\E[X]}^2],
  \end{align*}
  where each step follows from the fact that $X$ and $X'$ are
  i.i.d. Similarly, we also have
  \begin{equation*}
    \E[\ltwo{\Pi_C(X)-\Pi_C(X')}^2] =
    2\E[\ltwo{\Pi_C(X)-\E[\Pi_C(X)]}^2].
  \end{equation*}
  The projection $\Pi_C$ is non-expansive~\cite[Chapter
    III.3]{HiriartUrrutyLe96}, so $\E[\ltwo{\Pi_C(X)-\Pi_C(X')}^2]\leq
  \E[\ltwo{X-X'}^2]$, proving the lemma.
\end{proof}

With Lemmas~\ref{lemma:proj-mean} and~\ref{lemma:proj-variance} in place, we
can now give a convergence guarantee for the
estimator~\eqref{eqn:truncated-mean-estimator} of the parameter $\theta=\E[X]$.
 Indeed, by the two lemmas, we
see that via a bias-variance decomposition
\begin{align}
  \E[\ltwos{\what{\theta} - \theta}^2]
  &= \E\Big[\ltwoBig{\left(\E[\truncate{X}{T}]- \E[X]\right)+\frac{1}{n}\sum_{i=1}^n\left(\truncate{X_i}{T}-\E[\truncate{X}{T}]\right)+W}^2 \Big] \nonumber \\
  & = \ltwo{\E[\truncate{X}{T}] - \E[X]}^2
  + \frac{1}{n}\E[\ltwo{\truncate{X}{T} - \E[\truncate{X}{T}]}^2]
  + \E[\ltwo{W}^2] \nonumber \\
  & \le
  \frac{\E[\ltwo{X}^\moment]^2}{(\moment - 1)^2 T^{2\moment-2}}
  + \frac{ r^2}{n}
  + \E[\ltwo{W}^2],
  \label{eqn:bias-variance}
\end{align}
where we have used Lemmas~\ref{lemma:proj-mean}
and~\ref{lemma:proj-variance}
and the fact that $\var(X) \le \E[\ltwo{X}^2] \le
r^2$ by assumption.

Now, for each of the privacy types, we evaluate the risk of the resulting
estimator when we perturb the mean of the truncated variables by $W$.  We
begin with $\diffp$-KL privacy (equation~\eqref{eqn:kl-private}). In this
case, we take $W \sim \normal\left(0, \frac{T^2}{n^2 \diffpkl} I_{d \times
  d}\right)$, and using the decomposition~\eqref{eqn:bias-variance}, the
rate of convergence is bounded by
\begin{equation*}
  \E[\ltwos{\what{\theta} - \E[X]}^2]
  \le \frac{r^2}{n} + \frac{r^{2\moment}}{(\moment - 1)^2 T^{2\moment - 2}}
  + \frac{T^2 d}{n^2 \diffpkl}.
\end{equation*}
Setting $T = (n^2 \diffpkl / d)^{1 / (2\moment)}$ to approximately minimize
the preceding expression, we obtain that
\begin{equation*}
  \E[\ltwos{\what{\theta} - \E[X]}^2]
  \lesssim \frac{r^2}{n} + r^2 \left(\frac{d}{n^2 \diffpkl}
  \right)^{\frac{k - 1}{k}}.
\end{equation*}

To obtain the results for $(\diffp, \delta)$-approximate differential privacy
and $\diffp$-differential privacy, we sample $W$ from a $\normal(0, \frac{T^2
  \log \frac{1}{\delta}}{n^2 \diffp^2} I_{d \times d})$ distribution,
which yields $(\diffp,\delta)$-approximate differential privacy
as noted previously, and that
\begin{equation*}
  \E[\ltwos{\what{\theta} - \E[X]}^2]
  \le \frac{r^2}{n} + \frac{r^{2 \moment}}{(\moment - 1)^2 T^{2 \moment - 2}}
  + \frac{T^2 d \log \frac{1}{\delta}}{n^2 \diffp^2}.
\end{equation*}
Choosing $T = (n^2 \diffp^2 / (d \log \delta^{-1}))^{1 / (2 \moment)}$
gives the second result of the proposition.

For the final result, we note that for a vector $W \in \R^d$ with independent
coordinates with densities $p(w) \propto \exp(-\kappa|w|)$, where
$\kappa = \diffp n / T \sqrt{d}$,
satisfies
\begin{equation*}
  \E[\ltwo{W}^2]
  = 2 \sum_{j=1}^d \frac{T^2 d}{\diffp^2 n^2}
  = 2 \frac{d^2 T^2}{n^2 \diffp^2},
\end{equation*}
so that as before, choosing $T = (n^2 \diffp^2 / d^2)^{1 / (2\moment)}$ gives
 the final result.
\qed

\bibliographystyle{abbrvnat}
\bibliography{bib}

\end{document}